\documentclass[12pt]{amsart}

\usepackage{amsfonts,amssymb,amscd,amsmath,enumerate,verbatim,color}
\usepackage[latin1]{inputenc}
\usepackage{amscd}
\usepackage{latexsym}

\usepackage{graphicx}

\usepackage{tikz}
\usepackage{mathptmx}
\usepackage{mathrsfs}
\usepackage{amsthm}

\usetikzlibrary{intersections, calc, arrows}

\title{Toric rings of perfectly matchable subgraph polytopes}
\author{Kenta Mori}
\date{}


\newtheorem{thm}{Theorem}[section]
\newtheorem{prop}[thm]{Proposition} 
\newtheorem{lemma}[thm]{Lemma}
\newtheorem{cor}[thm]{Corollary}
\numberwithin{thm}{section}  
\numberwithin{equation}{section} 

\theoremstyle{definition}
 
\newtheorem{example}[thm]{Example}

\newtheorem{remark}[thm]{Remark}

\def\Pc{{\mathcal P}}
\def\NZQ{\mathbb}               
\def\NN{{\NZQ N}}

\def\ZZ{{\NZQ Z}}
\def\RR{{\NZQ R}}

\def\ab{{\mathbf a}}
\def\bb{{\mathbf b}}

\def\eb{{\mathbf e}}

\def\ub{{\mathbf u}}
\def\wb{{\mathbf w}}
\def\xb{{\mathbf x}}
\def\yb{{\mathbf y}}

\def\opn#1#2{\def#1{\operatorname{#2}}} 
\opn\gr{gr}

\def\Pc{{\mathcal P}}

\def\xb{{\mathbf{x}}}

\def\int{{\rm int}}
\def\stab{{\rm Stab}}

\def\PMSG{{{\mathcal P}_G}}
\def\PMSGd{{{\mathcal P}_{G'}}}

\begin{document}

\maketitle

\begin{center}
{\scriptsize Department of Mathematical Sciences,
	Graduate School of Science and Technology\\
	Kwansei Gakuin University,
	Sanda, Hyogo 669-1337, Japan\\
k-mori@kwansei.ac.jp
}
\end{center}

\begin{abstract}
The perfectly matchable subgraph polytope of a graph is a (0,1)-polytope associated with the vertex sets of matchings in the graph.
In this paper, we study algebraic properties (compressedness, Gorensteinness) 
of the toric rings of perfectly matchable subgraph polytopes. 
In particular, we give a complete characterization of a graph whose perfectly matchable subgraph polytope is compressed.
\end{abstract}

\section{Introduction}

A \textit{lattice polytope} $\mathscr{P}\subset\mathbb{R}^n$ is a convex polytope such that any vertex of $\mathscr{P}$ belongs to $\mathbb{Z}^n$.
Let $K[\mathbf{x}^{\pm1},s]=K[x^{\pm1}_1,\dots, {x}^{\pm1}_n,s]$ be a Laurent polynomial ring in $n+1$ variables over a field $K$.
For a lattice point $\alpha=(\alpha_1,\dots,\alpha_n)\in\ZZ^n$, we define $\xb^{\alpha}=x^{\alpha_1}_1\cdots x^{\alpha_n}_n\in K[\mathbf{x}^{\pm1},s]$.
If $\mathscr{P}\cap\mathbb{Z}^n=\{\ab_1,\dots,\ab_m\}$, then the \textit{toric ring} $K[\mathscr{P}]$ of $\mathscr{P}$ is the $K$-subalgebra of $K[\mathbf{x}^{\pm1},s]$ generated by the monomials $\mathbf{x}^{\mathbf{a}_1}s,\dots,\mathbf{x}^{\mathbf{a}_m}s\in K[\mathbf{x}^{\pm1},s]$. 
Furthermore, the \textit{toric ideal} $I_{\mathscr{P}}$ is the defining ideal of $K[\mathscr{P}]$, i.e., the kernel of a surjective ring homomorphism $\pi:K[y_1,\dots,y_m]\rightarrow K[\mathscr{P}]$ defined by $\pi(y_i)={\mathbf{x}^{\mathbf{a}_i}}s$ for $i=1,2,\dots,m$. It is known that $I_\mathscr{P}$ is generated by homogeneous binomials.
See, e.g., \cite{binomialideals, Stu} for details.

Compressed polytopes were defined by Stanley \cite{Stan} and have been studied from the viewpoint of polyhedral combinatorics, statistics, and optimization. 
A lattice polytope $\mathscr{P}$ is called \textit{compressed} if
the initial ideal of $I_\Pc$ is generated by squarefree monomials with respect to any reverse lexicographic order \cite{Sul}.
It is known that \cite[Corollary~8.9]{Stu}
the initial ideal of $I_\Pc$ is generated by squarefree monomials if and only if the
corresponding triangulation of $\mathscr{P}$ using only the lattice points in $\mathscr{P}$ is unimodular.
Hence $\Pc$ is compressed if and only if every pulling triangulation of $\mathscr{P}$ using only the lattice points in $\mathscr{P}$ is unimodular. 
Sullivant \cite{Sul} proved that a lattice polytope is compressed
if and only if it is \textit{$2$-level}, which is important in optimization theory.
For example, the convex polytope of all $n \times n$ doubly stochastic matrices,
hypersimplices, the order polytopes of finite posets, 
edge polytopes of bipartite graphs and complete multipartite graphs, and
the stable set polytopes of perfect graphs
are compressed.

On the other hand, $\Pc \subset \RR^n$ is said to be {\it normal} if $K[\Pc]$ is a normal semigroup ring.
It is known that 
\begin{itemize}
\item
$\Pc$ is normal if and only if
every vector in $k \Pc \cap L_\Pc$ is a sum of $k$ vectors from $\Pc \cap \ZZ^n$, 
where $L_\Pc$ is the sublattice of $\ZZ^n$ spanned by $\Pc\cap\ZZ^n$;
\item
$\Pc$ is normal if there exists a monomial order such that 
the initial ideal of $I_\Pc$ is generated by squarefree monomials. 
In particular, $\Pc$ is normal if $\Pc$ is compressed. 
\end{itemize}
A lattice polytope $\Pc \subset \RR^n$ has the \textit{integer decomposition property} (IDP) if every vector in $k \Pc \cap\ZZ^n$ is a sum of $k$ vectors from $\Pc\cap\ZZ^n$.
In particular, $\Pc$ is normal if $\Pc$ has IDP.
However, the converse does not hold in general.

A lattice polytope $\Pc \subset \RR^n$ is said to be \textit{reflexive} if ${\bf 0}$ is the unique lattice point in its interior and the dual polytope
$$
\Pc^* := \{ x\in\RR^n : x \cdot y\leq 1 \mbox{ for any } y\in \Pc\}
$$
is again a lattice polytope. 
Here $x\cdot y$ is the inner product of $x$ and $y$.
Note that each vertex of $\Pc^*$ corresponds to a facet of $\Pc$.
Two lattice polytopes $\Pc \subset\RR^n$ and $\Pc'\subset\RR^{n'}$ are said to be \textit{unimodularly equivalent} if there exists an affine map from the affine span 
$${\rm aff} (\Pc)= \left\{ \sum_{i=1}^r \lambda_i \alpha_i : 1 \le r \in \ZZ, \alpha_i \in \Pc, \lambda_i \in \RR, \sum_{i=1}^r \lambda_i=1 \right\}$$
of $\Pc$ to the affine span ${\rm aff} (\Pc')$ of $\Pc'$ that maps $\ZZ^n\cap{\rm aff}(\Pc)$ bijectively onto $\ZZ^{n'}\cap{\rm aff}(\Pc')$ and that maps $\Pc$ to $\Pc'$.
A lattice polytope $\Pc \subset \RR^n$ of dimension $n$ is called \textit{Gorenstein of index $\delta$} if $\delta\Pc=\{\delta a:a\in\Pc\}$ is unimodularly equivalent to a reflexive polytope. In particular, a reflexive polytope is Gorenstein of index 1.
Note that a lattice polytope $\Pc \subset \RR^n$ of dimension $n$ is Gorenstein of index $\delta$ if and only if there exist 
a positive integer $\delta$ and a lattice point $\alpha \in\delta (\Pc\setminus\partial\Pc)\cap\ZZ^n$ such that $\delta \Pc - \alpha$ is a reflexive polytope, where $\partial\Pc$ is the boundary of $\Pc$.
Reflexive polytopes are related to mirror symmetry and studied in many areas of mathematics. 
They are key combinatorial tools for constructing topologically mirror-symmetric pairs of Calabi-Yau varieties, as shown by Batyrev \cite{Batyrev}.
It is known that a lattice polytope $\Pc$ is Gorenstein if and only if  
the Ehrhart ring 
$$K[{\xb}^{\alpha} s^{m}:\alpha\in mP\cap\ZZ^{n},m\in\ZZ_{\ge 0}]\subset K[\mathbf{x}^{\pm1},s]$$
of $\Pc$ is Gorenstein.
On the other hand, the Ehrhart ring of $\Pc$
coincides with the toric ring of $\Pc$
if and only if $\Pc$ has IDP.

In the present paper, we study conditions for perfectly matchable subgraph polytopes to be compressed or Gorenstein.
Let $G =(V,E)$ be a graph on the vertex set $V=[n]:=\{1,2,\dots,n\}$ and the edge set $E$.
Throughout this paper, all graphs are assumed to be finite and simple.
A $k$-matching of $G$ is a set of $k$ pairwise non-adjacent edges of $G$.
If a matching $M$ includes all vertices of $G$, then $M$ is called a \textit{perfect matching}.
We say that $S\subset V$ induces a \textit{perfectly matchable subgraph} of $G$ if the induced subgraph $G[S]$ of $G$ on the vertex set $S$ has a perfect matching.
Let $\mathscr{W}(G)$ be the set of all such subsets of $V$, and adopt the convention that $\emptyset\in\mathscr{W}(G)$, i.e., that the empty subgraph is perfectly matchable. 
Given a subset $S \subset V$, let $\rho(S) = \sum_{i \in S} \mathbf{e}_i \in \RR^n$,
where $\mathbf{e}_i$ is the $i$th unit vector in $\RR^n$.
In particular, $\rho(\emptyset) $ is the zero vector.
The \textit{perfectly matchable subgraph polytope} of $G$, denoted by $\PMSG$, is the convex hull of 
$\{ \rho(S) \in \RR^n :  S \in\mathscr{W}(G) \}$.

The perfectly matchable subgraph polytope
of a graph is defined in \cite{PMS}.
The motivation of their study on  perfectly matchable subgraph polytopes is to solve optimization problems that arise in practice
(e.g., a bus driver scheduling problem).
In optimization theory, compressed polytopes are important
since 
semidefinite programming relaxations are very efficient for 
compressed polytopes (see, e.g, \cite{GPT}).

Recently, perfectly matchable subgraphs of graphs
appear in the study of $h^*$-polynomials of lattice polytopes.
A graph $G$ is called \textit{$k$-partite} if 
the vertex set of $G$ can be partitioned into $k$ different independent sets $V_1 , \dots , V_k$. 
When $k=2$, it is called a \textit{bipartite graph}.
If $G$ is a bipartite graph with a partition $V_1\cup V_2=[n]$, let $\widehat{G}$ be a connected bipartite graph on $[n+2]$ whose edge set is
$E(\widehat{G})=E(G)\cup\{\{i,n+1\}:i\in V_1\}\cup\{\{j,n+2\}:j\in V_2\cup\{n+1\}\}$.
It is known \cite[Proposition~3.4]{OT1} that 
$$I_{\widehat{G}}(x)=\sum_{S \in \mathscr{W}(G)} x^{|S|},$$
where $I_{\widehat{G}}(x)$ is the interior polynomial of $\widehat{G}$
that is introduced by K\'{a}lm\'{a}n \cite{kal} 
as a version of the Tutte polynomials for hypergraphs. 
It was shown \cite{kalpo} that the $h^*$-polynomial of
the edge polytope of a bipartite graph $G$ coincides with the interior polynomial $I_G(x)$ of a hypergraph induced by $G$.
Using these facts, several results on 
$h^*$-polynomials of several important classes of
lattice polytopes are obtained \cite{DK,OT1,OT2,OT3}.

If $G$ is the disjoint union of graphs $G_1$ and $G_2$,
then $\PMSG$ is the product of ${\mathcal P}_{G_1}$ and ${\mathcal P}_{G_2}$.
Hence, $\PMSG$ is compressed (resp. Gorenstein) if and only if both ${\mathcal P}_{G_1}$ and ${\mathcal P}_{G_2}$
are compressed (resp. Gorenstein).
Thus, when we are studying such properties, we may assume that $G$ is connected.

 The first main result of the present paper is a complete characterization of compressed 
perfectly matchable subgraph polytopes.
A \textit{complete $k$-partite graph} denoted by $K_{|V_1|,\dots,|V_k|}$ is a
$k$-partite graph
with a partition $V_1 \cup \cdots \cup V_k$ of its vertex set such that $\{s,t\}$ is an edge of $G$ for 
any $s\in V_i$, any $t\in V_j$, and any $1\le i < j\le k$.
A complete $n$-partite graph all of whose independent sets $V_i$ have only one vertex is called a \textit{complete graph}, and denoted by $K_n$.
A vertex $v$ of a connected graph $G$ is called a \textit{cut vertex} if the graph obtained by the removal of $v$ from $G$ is disconnected. Given a graph $G$, a \textit{block} of $G$ is a maximal connected subgraph of $G$ without cut vertices.

\begin{thm}
\label{compressed}
Let $G$ be a connected graph.
Then $\PMSG$ is compressed
if and only if all blocks of $G$ are complete bipartite graphs except for at most one block, which is either $K_4$ or $K_{1,1,n}$.
\end{thm}

In particular, if $\PMSG$ is compressed, then the line graph of $G$
is perfect by Proposition~\ref{lineperfect}.

The second main result of the present paper is a characterization of Gorenstein 
perfectly matchable subgraph polytopes of bipartite graphs.
For any $S\subset \mathit{V}$, let 
$\Gamma(S)$ denote the subset of $V \setminus S$ that consists of vertices
adjacent to at least one vertex in $S$. 
Theorem~\ref{secondmain} follows from 
Proposition~\ref{IDP},
Theorem \ref{2Gor} and Corollary \ref{pmsedgeGor}.

\begin{thm}
\label{secondmain}
Suppose that a connected bipartite graph
$G = (V_1 \cup V_2 , E)$ has a vertex $v$ with $\deg(v) \ge 2$ such that $v$ is not a cut vertex.
Then the following conditions are equivalent{\rm :}
\begin{itemize}
    \item[{\rm (i)}] 
$K[\PMSG]$ is Gorenstein{\rm ;}
        \item[{\rm (ii)}] 
$\PMSG$ is Gorenstein{\rm ;}
    \item[{\rm (iii)}] 
    $\PMSG$ is Gorenstein of index $2${\rm ;}
    \item[{\rm (iv)}] 
$G$ has a perfect matching and, for any subset $\emptyset \ne S\subsetneq V_1$ such that
$G[S\cup\Gamma(S)]$ and $G[(V_1\setminus S)\cup(V_2\setminus \Gamma(S))]$ are connected,
    we have $|S| +1 = |\Gamma(S)|$.
    
\end{itemize}
Moreover, if $G$ is 2-connected, then the above conditions are equivalent to
\begin{itemize}
    \item[{\rm (v)}] 
the edge polytope of $G$ is Gorenstein.

\end{itemize}

\end{thm}

The third main result of the present paper is a complete characterization of Gorenstein 
perfectly matchable subgraph polytopes of pseudotrees.
A connected graph which has at most one cycle is called a \textit{pseudotree}. 
A \textit{regular graph} is a graph whose vertices have the same degree.
A \textit{bidegreed graph} is a graph with two different vertex degrees.
For example, a path $P_n$ and a star graph $K_{1,n-1}$ are bidegreed if $n \ge 3$.

\begin{thm}
\label{pseudotreeGorensteinODD}
Let $G$ be a pseudotree on the vertex set $V$.
Then $K[\PMSG]$ is Gorenstein
if and only if 
$G$ satisfies one of the following{\rm :}
\begin{itemize}
    \item[{\rm (i)}] $G$ is $K_1$, $K_2$, or a bidegreed tree{\rm ;}
    \item[{\rm (ii)}] $G=C_5${\rm ;}
    \item[{\rm (iii)}] $G$ has a triangle $C$ and
$$\begin{cases}
\deg(v) \in \{2,3\} & \mbox{if }v \in V(C)\\
\deg(v) \in \{1,3\} & \mbox{otherwise{\rm ;}}
\end{cases}    $$
    \item[{\rm (iv)}] $G$ has an even cycle $C$, and there exists an integer $\delta \ge 2$ such that
$$\begin{cases}
\deg (v) = \delta & \mbox{if }v \in V(C)\\
\deg (v) \in \{1,\delta-1\} & \mbox{otherwise.}
\end{cases} $$   
 %
\end{itemize}
\end{thm}




The relationships between toric rings of perfectly matchable subgraph polytopes and 
toric rings of other polytopes play important roles in this paper.
In \cite{PMS}, it was pointed out that 
$\PMSG$ is unimodulary equivalent to a base polytope of a transversal matroid if $G$ is bipartite. 
Let $G$ be a graph with the edge set $E =\{e_1,\dots,e_m\}$,
and let
$$
A_G= (\rho(e_1) , \dots, \rho(e_m))
$$
be the matrix associated with the edge polytope 
${\rm Ed}(G)$ of $G$.
It then follows that 
\begin{equation}
\PMSG \cap \ZZ^n =\{ A_G \ x : x \in {\rm Stab}(L(G)) \cap \ZZ^m \},\label{PMSstable}
\end{equation}
where ${\rm Stab}(L(G))$ is the stable set polytope of the line graph of $G$
(definitions are explained later).
From (\ref{PMSstable}), it is easy to see that
the edge polytope of $G$ is normal if $\PMSG$ is normal
(Corollary \ref{normaltonormal}).
There are many research 
on the edge polytopes and the stable set polytopes
from the point of view of not only discrete geometry but also combinatorial commutative algebra.
The study of perfectly matchable subgraph polytopes is expected to contribute the study of these polytopes.

The present paper is organized as follows.
In Section $2$, we introduce relationships between $\PMSG$ and other polytopes.
In Section $3$, in order to prove the main theorems, we examine inequalities which are facet-inducing for $\PMSG$. 
In Section $4$, we give a proof for Theorem~\ref{compressed}.
Finally, in Section $5$, we give a proof for Theorems
\ref{secondmain} and
\ref{pseudotreeGorensteinODD}.\\[10pt]

\section{Relationships with toric rings of other polytopes}

In this section, we introduce relationships between toric rings of
$\PMSG$ and toric rings  of other polytopes, i.e.,
base polytopes of matroids, stable set polytopes, and edge polytopes.

Let $G=(V,E)$ be a graph.
Recall that the perfectly matchable subgraph polytope $\PMSG$ of $G$ is the convex hull of 
$\{ \rho(S) \in \RR^n :  S \in\mathscr{W}(G) \}$,
where $\mathscr{W}(G)$ is the set of all subsets of $V$ which induce perfectly matchable subsets of $G$.

\begin{example}
Let $G$ be a cycle $C_4$ of length 4.
Then the set $\mathscr{W}(G)$ associated with $G$ is 
$$\mathscr{W}(G) = \{\emptyset, \{1,2\},\{2,3\},\{3,4\},\{1,4\},\{1,2,3,4\}\}.$$
Note that matchings $\{\{1,2\},\{3,4\}\}$ and $\{\{1,4\},\{2,3\}\}$ are associated with
the same set $\{1,2,3,4\}$.
The perfectly matchable subgraph polytope $\PMSG \subset \RR^4$ of $G$ 
is the convex hull of the column vectors of the matrix
$$\begin{pmatrix}
0&1&0&0&1&1 \\
0&1&1&0&0&1 \\
0&0&1&1&0&1 \\
0&0&0&1&1&1 
\end{pmatrix}.$$
Then $\PMSG$ is a 3-dimensional polytope which is compressed and Gorenstein.
\end{example}

It is known \cite{PMS, PMS2} that, if $G$ is a connected graph on the vertex set $\{1,2,\dots,n\}$, then
$$
\dim \PMSG = 
\left\{
\begin{array}{cc}
n-1   & \mbox{if } G \mbox{ is bipartite,} \\
n & \mbox{otherwise.}
\end{array}
\right.
$$
If $G=(V, E)$ is a bipartite graph on the vertex set $V=\{1,2,\dots,n\}=V_1\cup V_2$, then $\PMSG$ is contained in the hyperplane
$$\{ x \in \RR^n :  x(V_1) = x(V_2) \},$$
where $x(V_i)$ denotes the sum $\sum_{j \in V_i} x_j$ for the vector $x = (x_1,\dots,x_n)$. 
If $G$ is the disjoint union of graphs $G_1$ and $G_2$,
then $\PMSG$ is the product of ${\mathcal P}_{G_1}$ and ${\mathcal P}_{G_2}$ 
and hence $\dim (\PMSG) = \dim ({\mathcal P}_{G_1}) + \dim ({\mathcal P}_{G_2})$.
If $G$ has $k$ connected components which are bipartite, then we have
$
\dim \PMSG = n- k.
$

\subsection{Base polytopes of transversal matroids}

Let $M$ be a matroid on a ground set $\{1,2,\dots,n\}$
 with the set of bases $\mathcal{B}$.
The {\it base polytope} $B(M)$ of $M$ is the convex hull of the set $\{ \rho(B) \ :\  B \in \mathcal{B} \}\subset \RR^n$.
In \cite{PMS}, it was pointed out that 
$\PMSG$ is unimodularly equivalent to
a base polytope of a transversal matroid if $G$ is bipartite. 
Since the base polytope of any matroid has IDP
\cite{White},
we have the following.

\begin{prop}
\label{bipartitenormal}
\label{IDP}
Let $G$ be a connected bipartite graph.
Then $\PMSG$ has IDP.    
In particular, $K[\PMSG]$ is Gorenstein if and only if $\PMSG$ is Gorenstein.
  \end{prop}



Note that 
Gorenstein base polytopes of matroids were studied in \cite{GMP}.

\subsection{Edge polytopes and stable set polytopes}
  
Let $G$ be a graph on the vertex set $V =\{1,2,\dots,n\}$ and the edge set $E =\{e_1,\dots,e_m\}$.
The \textit{edge polytope} ${\rm Ed}(G)$ of $G$ is the convex hull of the set
$$\{ \rho(e_1),\dots, \rho(e_m)\}.$$
Note that 
compressed (resp. Gorenstein) edge polytopes were studied in \cite{Ocomedge, OHcompressed01} (resp. \cite{OH}).
We say that a graph $G$ satisfies the \textit{odd cycle condition}
if, for any two odd cycles $C_1$ and $C_2$ in 
the same connected component of $G$ without common vertices, 
there exists an edge $\{i,j\}$ of $G$
such that $i \in V(C_1)$ and $j \in V(C_2)$.

\begin{prop}[\cite{OHnormal, SVV}]
\label{epnormal}
Let $G$ be a graph.
Then ${\rm Ed}(G)$ is normal if and only if $G$ satisfies the odd cycle condition.
  \end{prop}

A finite subset $S\subset V$ is called \textit{stable} in $G$ if none of the edges of $G$ is a subset of $S$. 
In particular, the empty set $\emptyset$ is stable. 
Let $S(G) = \{S_1,\dots,S_t\}$ denote the set of all stable sets of $G$.
The \textit{stable set polytope} of $G$ is the convex hull of the set $\{\rho(S_1),\dots,\rho(S_t)\}$,
denoted by $\stab(G)$.
It is known (e.g., \cite{OHcompressed02}) that $\stab(G)$ is compressed if and only if $G$ is perfect.
Moreover, it is known \cite[Theorem 1.2 (b)]{OH} that, for any perfect graph $G$, $\stab(G)$ is Gorenstein if and only if
all maximal cliques of $G$ have the same cardinality.
The {\it line graph} $L(G)$ of $G=(V,E)$ is a graph on the vertex set $E$
with the edge set $\{ \{e,e'\} : e,e' \in E , e \cap e' \neq \emptyset\} $.

For $S\in\mathscr{W}(G)$, we have
$\rho(S)=\rho(e_{i_1})+\dots+\rho(e_{i_k}) \in \RR^n,$
where $\{e_{i_1},\dots, e_{i_k}\} \subset E$
is a matching of $G$.
Note that
$\{e_{i_1},\dots, e_{i_k}\} \subset E$
is a matching of $G$ if and only if
$S'=\{e_{i_1},\dots, e_{i_k}\}\subset V(L(G))$ is stable in $L(G)$. 
Since $\rho(S')= \eb_{i_1}+\dots+\eb_{i_k} \in \RR^m$, we have 
$$\rho(S)=A_G \ \rho(S'),$$
where
$
A_G= (\rho(e_1) , \dots, \rho(e_m))
$
is the vertex-edge incidence matrix of $G$.
Thus
%
\begin{equation}
\PMSG \cap \ZZ^n =\{ A_G \ x : x \in {\rm Stab}(L(G)) \cap \ZZ^m \}.\label{latticePS} 
\end{equation}
In addition, we have ${\rm Ed}(G) \subset \PMSG$.
In such a case, the following holds in general.

\begin{prop}
\label{normalfact}
Let $\Pc \subset \Pc' \subset \RR^n$ be lattice polytopes
such that
\begin{eqnarray*}
    \Pc \cap \ZZ^n &=& \{\ab_1, \ab_2, \dots, \ab_m\},\\
\Pc' \cap \ZZ^n  &=&\{ (\ab_1, \ab_2, \dots, \ab_m) \ x : x \in Q \},
\end{eqnarray*}
where $Q \subset \ZZ_{\ge 0}^m$.
Suppose that there exists $\wb \in \RR^n$ such that $\ab_i \cdot \wb = 1$
for any $1 \le i \le m$.
Then
$\Pc$ is normal if $\Pc'$ is normal.
\end{prop}

\begin{proof}
Let $\alpha \in k \Pc \cap L_\Pc$.
Since $\Pc \subset \Pc'$, $\alpha$ belongs to $k \Pc' \cap L_{\Pc'}$.
The normality of $\Pc'$ guarantees that 
$\alpha = \alpha_1 + \dots + \alpha_k$ where $\alpha_i \in \Pc' \cap \ZZ^n$
for each $i$.
Then  $\alpha_i = (\ab_1, \ab_2, \dots, \ab_m) x$ 
for some nonnegative integer vector $x \in Q$.
Hence each $\alpha_i$ is a sum of vectors from $\Pc \cap \ZZ^n$
if $x$ is not zero.
Thus we have $\alpha= \ab_{j_1} + \cdots + \ab_{j_\ell}$
for some $1 \le j_1,\dots,j_\ell \le m$.
Since
there exists $\wb \in \RR^n$ such that $\ab_i \cdot \wb = 1$
for any $1 \le i \le m$,
it follows that $k = \wb \cdot \alpha = \wb \cdot \ab_{j_1} + \cdots + \wb \cdot \ab_{j_\ell} = \ell$.
Thus $\alpha$ is a sum of $k$ vectors from $\Pc \cap \ZZ^n$, as desired.
\end{proof}

Since $\rho(e_i) \cdot \wb = 1$ ($1 \le i \le m$) for $\wb = (1/2,\dots,1/2)$,
we have the following.

\begin{cor}
    \label{normaltonormal}
Let $G$ be a graph.
If $\PMSG$ is normal, then
${\rm Ed}(G)$ is normal
(i.e., $G$ satisfies the odd cycle condition).
\end{cor}

Given a lattice polytope $\Pc \subset \RR^n$, 
let $A = (\ab_1,\dots,\ab_m)$ where 
 $\Pc \cap \ZZ^n = \{\ab_1,\dots,\ab_m\}$.
It is known \cite{Stu} that 
the toric ideal $I_\Pc$ of $\Pc$ is generated by binomials
$\yb^\ab - \yb^\bb \in K[y_1,\dots,y_m]$ such that 
$A(\ab - \bb) ={\bf0}$ and $\deg \yb^\ab =\deg \yb^\bb $.


\begin{prop}
\label{doukei}
Let $G$ be a pseudotree which has no even cycles. 
the toric ring $K[\PMSG]$ of $\PMSG$ is 
isomorphic to the toric ring $ K[\stab (L(G))]$ of
$\stab (L(G))$.
\end{prop}

\begin{proof}
Let $G =(V,E)$ be a graph on the vertex set $V =\{1,2,\dots,n\}$
and the edge set $E =\{e_1,\dots,e_m\}$.
Let
$A_G= (\rho(e_1) , \dots, \rho(e_m))$
be the vertex-edge incidence matrix of $G$
and let $B_G = (\rho(S_1), \dots, \rho(S_r))$
where $\{S_1,\dots, S_r\}$ is the set of all stable sets of $L(G)$.

If $G$ is not a tree, then
$n=m$ and hence
$A_G$ is an $m \times m$ matrix.
It is known \cite[Lemmas 5.5 and 5.6]{binomialideals} that $A_G$ is a regular matrix if and only if  $G$ is a pseudotree which has no even cycles.
If $G$ is a tree, then $n=m+1$ and 
$A_G$ is an $(m+1) \times m$ matrix of rank $m$.

Since the rank of the $n \times m$ matrix $A_G$ is $m$, for any $\ub \in \RR^r$,
$A_G B_G \ub = {\bf 0}$ if and only if $B_G \ub = {\bf 0}$.
From (\ref{latticePS}),
$\PMSG \cap \ZZ^n =\{ A_G \rho(S_1), \dots,A_G \rho(S_r)\}.
$
Hence we have
\begin{eqnarray*}
\yb^\ab - \yb^\bb \in I_\PMSG
&\Leftrightarrow&
A_G B_G (\ab - \bb) = {\bf 0} \mbox{ and } \deg \yb^\ab = \deg  \yb^\bb \\
&\Leftrightarrow    & B_G (\ab - \bb) = {\bf 0} \mbox{ and } \deg \yb^\ab = \deg  \yb^\bb \\
&\Leftrightarrow    &
\yb^\ab - \yb^\bb \in I_{{\rm Stab}(L(G))}.
\end{eqnarray*}
Thus $K[\PMSG] \simeq K[\yb]/ I_\PMSG = K[\yb]/ I_{\rm Stab}(L(G)) \simeq K[\stab (L(G))]$ as desired.
\end{proof}

\begin{prop}
\label{pseudotree normal}
Let $G$ be a pseudotree. Then $\PMSG$ is normal.
\end{prop}

\begin{proof}
From Proposition \ref{bipartitenormal}, we may assume that $G$ is not bipartite. 
Then, by Proposition~\ref{doukei}, $K[\PMSG] \cong K[\stab (L(G))]$.
It is known \cite[Theorem 8.1]{almostbip} that 
the toric ideal of the stable set polytope of an almost bipartite graph has a squarefree quadratic initial ideal. 
It then follows that $\stab (C_{2n+1})$ is normal.
In addition, $\stab (K_{m})$ is a simplex and hence normal.
It is known \cite[Proposition 1]{MOS} that 
the stable set polytope of 
the clique-sum of simple graphs $G_1$ and $G_2$ is normal if and only if both $\stab(G_1)$ and $\stab(G_2)$ are normal.
Since $L(G)$ is a clique-sum of an odd cycle $C_{2n+1} = L(C_{2n+1})$ and some cliques, ${\rm Stab}(L(G))$ is normal.
\end{proof}

\section{Facets of perfectly matchable subgraph polytopes}

In this section, we introduce inequalities for the facets of $\PMSG$ given in \cite{PMS, PMS2}. 
We will see that these inequalities depend on whether $G$ is bipartite.
Let
\begin{center}
$\mathscr{T}=\{\mathit{X}\subset\mathit{V} :$ each component of $G[X]$ has an odd number of vertices\}.
\end{center}
For any $\mathit{A}\subset \mathit{V}$, let 
$\Gamma(A)$ denote the subset of $V \setminus A$ that consists of vertices
adjacent to at least one vertex in $A$. 
For any $S\subset V$, let $\theta(S)$ be the number of connected components of the induced subgraph $G[S]$.

\begin{prop}[\cite{PMS2}]
\label{inequalityANY}
Let $G=(V,E)$ be a graph.
Then $\PMSG$ is a set of vectors $x \in \RR^V$ such that
\begin{eqnarray}
0\leq x(v) \leq 1 \mbox{ for all }  v \in V\label{zeroone1}\\
x(S)-x(\Gamma(S))\leq|S|-\theta(S) \label{facet1}
\end{eqnarray}
for all $S\in\mathscr{T}$ such that every component of $G[S]$ consists of a single vertex or else is a nonbipartite graph with an odd number of vertices.
\end{prop}
A graph $G=(V,E)$ is called \textit{critical} (or \textit{hypomatchable}) if, for every $v\in V$,
$G\setminus\{v\}$ has a perfect matching. 
A critical graph is a nonbipartite graph with an odd number of vertices.

\begin{prop}[\cite{PMS2}]
\label{nonbipartitefacet}
Let $G$ be a nonbipartite graph.
For $S\in\mathscr{T}$, the inequality (\ref{facet1}) is facet-inducing for $\PMSG$ if and only if 
$S$ satisfies the following conditions{\rm :}
\begin{itemize}
\item[{\rm (i)}] every component of $G[S]$ is critical{\rm ;}
\item[{\rm (ii)}] every component of $G\setminus(S\cup\Gamma(S))$ is nonbipartite{\rm ;}
\item[{\rm (iii)}] the graph obtained from $G[S\cup\Gamma(S)]$ by deleting all edges with both ends in $\Gamma(S)$ is connected.
\end{itemize}
\end{prop}

Remark that, if $|S|=1$, then $S$ satisfies conditions (i) and (iii).

\begin{prop}[\cite{PMS}]
\label{bipartitefacetfor}
Let $G=(V,E)$ be a bipartite graph on the vertex set $V=V_1\cup V_2$.
Then $\PMSG$ is a set of vectors $x \in \RR^V$ such that
\begin{eqnarray}
0\leq x(v) \leq 1 \mbox{ for all }  v \in V, \label{zeroichi}\\ 
x(S)-x(\Gamma(S))\leq0 \mbox{ for all } \emptyset \ne S\subset V_1, \label{facet2}\\
x(V_1)-x(V_2)=0. \label{hitoshii}
\end{eqnarray}

\end{prop}

\begin{prop}[\cite{PMS}]\label{bipartitefacet}
Let $G=(V_1\cup V_2,E)$ be a connected bipartite graph.
Then the following are true: 
\begin{itemize}
    \item[{\rm (a)}]
The inequality $0\leq x(v)$ is facet-inducing for $\PMSG$ if and only if
$v$ is not a cut vertex. (Note that every vertex of degree one 
    is not a cut vertex.)
    \item[{\rm (b)}]
Suppose that $G$ has at least two edges.
Then the inequality $x(v) \leq 1$ is facet-inducing for $\PMSG$ if and only if $\deg(v) \ge 2${\rm ;}
    \item[{\rm (c)}]
    For any $\emptyset \ne S\subsetneq V_1$, the inequality (\ref{facet2}) is facet-inducing for $\PMSG$ if and only if both $G[S\cup\Gamma(S)]$ and $G[(V_1\setminus S)\cup(V_2\setminus \Gamma(S))]$ are connected. 
\end{itemize}
\end{prop}

\begin{remark}
    Suppose that both $G[S\cup\Gamma(S)]$ and $G[(V_1\setminus S)\cup(V_2\setminus \Gamma(S))]$ are connected for $S \subsetneq V_1$.
Let $S' = V_2 \setminus \Gamma (S)$.
Then we have $\Gamma(S') = V_1\setminus S $.
If (\ref{hitoshii}) holds, then we have 
$x(S)-x(\Gamma(S)) = x(S')-x(\Gamma(S')).$
\end{remark}

\section{Compressed perfectly matchable subgraph polytopes}

In this section, we prove Theorem \ref{compressed}. 
The following proposition is due to Sullivant \cite[Theorem 2.4]{Sul}
(and also appeared in Haase's dissertation \cite{Haase}).

\begin{prop}[\cite{Sul}]
\label{sullivantcompressed}
Let $\mathscr{P}$ be a lattice polytope having the irredundant linear description 
$\mathscr{P}=\{\xb \in\mathbb{R}^n : \ab_i \cdot \xb \geq b_i, \ i=1,\dots,s\},$
where $\ab_i \in \RR^n$ for $i=1,2,\dots,s$.
In addition, let $\mathscr{L} \subset \ZZ^n$ be a lattice spanned by $\mathscr{P} \cap \ZZ^n$.
Then $\mathscr{P}$ is compressed if and only if, for each $i$, there is at most one nonzero $m_i \in \RR$ such that 
    $\{\xb \in\mathscr{L}: \ab_i \cdot \xb = b_i+m_i\}\cap \mathscr{P} \ne \emptyset.$
\end{prop}

Let $G'$ be an induced subgraph of a graph $G$.
Then $\PMSGd$ is a face of $\PMSG$.
It is known that every face of a compressed polytope is compressed.
Hence, we have the following immediately.

\begin{lemma}
\label{inducedcompressed}
Let $G'$ be a connected graph such that $\PMSGd$ is not compressed.
If a graph $G$ has $G'$ as an induced subgraph, then $\PMSG$ is not compressed.
\end{lemma}

The following fact is known in graph theory.

\begin{prop}[\cite{Maf, Tro}]
\label{lineperfect}
Let $G$ be a graph. Then the following conditions are equivalent{\rm :}
\begin{itemize}
    \item[\rm (i)] The line graph $L(G)$ of $G$ is perfect{\rm ;} 
    \item[\rm (ii)] $G$ has no odd cycle of length $\ge 5$ as a subgraph{\rm ;}
    \item[\rm (iii)] Each block of $G$ is either a bipartite graph, $K_4$, or $K_{1,1,n}$.
\end{itemize}

\end{prop}

It is known that $\stab (G)$ is compressed
if and only if $G$ is perfect.

\begin{lemma}
\label{nooddcycle}
Let $G$ be a connected graph.
If $\PMSG$ is compressed, then 
$L(G)$ is perfect and hence $\stab{(L(G))}$ is compressed.
\end{lemma}

\begin{proof}
Suppose that $L(G)$ is not perfect.
From Proposition \ref{lineperfect}, $G$ has an odd cycle $C_{2n+1}$ with $n\geq2$ as a subgraph.
Let $S=V(C_{2n+1})$ and $G'=G[S]$. 
From Lemma \ref{inducedcompressed}, it is enough to prove that 
$\PMSGd$ arising from the induced subgraph $G'$ of $G$ is not compressed.
By assumption, $G'\setminus(S\cup\Gamma(S))$ is empty.
Since $C_5$ is a Hamiltonian cycle of $G'$, the graph $G'=G'[S]$ is critical.
Moreover, $\Gamma(S)=\emptyset$
and $G'[S\cup\Gamma(S)]=G'$ is connected.
Hence,
$$x(S) \le |S|-\theta(S)=2n+1-1=2n$$
is facet-inducing for $\PMSGd$ from Proposition \ref{nonbipartitefacet}.
Then there exist $n+1 \ge 3$ kinds of values for $x(S)$ 
with $x \in \PMSGd\cap \ZZ^{2n+1}$.
In fact, we have
$$
\left\{x(S): x \in \PMSGd \cap \ZZ^{2n+1} \right\} = \{0,2,\ldots,2n\}.
$$
From Proposition \ref{sullivantcompressed}, $\PMSGd$ is not compressed.
\end{proof}

\begin{lemma}
\label{completebipartite}
Let $G$ be a connected graph.
Suppose that $\PMSG$ is compressed.
Then for any even cycle $C$ in $G$ of length $2n \ge 6$,
the induced subgraph
$G[V(C)]$ is a complete bipartite graph $K_{n,n}$.
\end{lemma}

\begin{proof}
Suppose that $\PMSG$ is compressed.
Let $C=(v_1,\dots,v_{2n})$ be an even cycle in $G$ of length $2n \ge 6$,
and let $G' = G[V(C)]$.
We prove the statement by induction on $n$.

From Lemma \ref{inducedcompressed}, $\PMSGd$ is compressed.
Note that $G'$ is a subgraph of a block of $G$.
From Proposition \ref{lineperfect} and Lemma \ref{nooddcycle}, $G'$ is a bipartite graph since neither $K_4$ nor $K_{1,1,n}$ has  an even cycle of length $\geq 6$.
Suppose that $G'$ is not a complete bipartite graph.
\bigskip

\noindent{\bf Case 1.} ($n=3$)
Suppose that $\{v_3,v_6\}$ is not an edge of $G'$. Then, for $S=\{v_3\}$, 
both $G[S\cup\Gamma(S)]=G[\{v_2,v_3,v_4\}]$ and $G[(V_1\setminus S)\cup(V_2\setminus \Gamma(S))=G[\{v_1,v_5,v_6\}]$ are connected. 
However, we have
$$
x(S)-x(\Gamma(S))
=
\left\{
\begin{array}{cl}
0     & \mbox{for } x =\rho(\emptyset) = {\bf 0}, \\
-1     & \mbox{for }  x = \rho(\{v_1, v_2\}),\\
-2     & \mbox{for } x = \rho(\{v_1,v_2, v_4,v_5\}).
\end{array}
\right.
$$
Hence, $\PMSGd$ is not compressed, a contradiction.
It follows that $C_6$ has all the chords $\{v_1,v_4\}$, $\{v_2,v_5\}$, and $\{v_3,v_6\}$.
Hence, $G'$ is a complete bipartite graph $K_{3,3}$.
\bigskip

\noindent{\bf Case 2.}
($n\ge 4$ and suppose that the statement is true for any even cycle
of length $\le 2n-2$)

Suppose that $\{v_3, v_{2k}\}$ for some $3 \le k \le n$ is not an edge of $G'$.
If $\{v_3, v_{2k'}\}$ is an edge of $G'$ for some $k'$,
then $v_3, v_{2k}$ are contained in an even cycle of length $2m$
with $6 \le 2m \le 2n-2$.
By the hypothesis of induction,  $\{v_3, v_{2k}\}$ is an edge of $G'$,
a contradiction. 
Thus, for any  $3 \le k' \le n$, $\{v_3, v_{2k'}\}$ is not an edge of $G'$.
Let $S=\{v_3\}$. 
Then both $G[S\cup\Gamma(S)]=G[v_2,v_3,v_4]$ and $G[(V_1\setminus S)\cup(V_2\setminus \Gamma(S))]=G[v_1,v_5,v_6, \dots, v_{2n}]$ are connected. 
However, we have
$$
x(S)-x(\Gamma(S))
=
\left\{
\begin{array}{cl}
0     & \mbox{for } x =\rho(\emptyset) = {\bf 0}, \\
-1     & \mbox{for }  x = \rho(\{v_1, v_2\}),\\
-2     & \mbox{for } x = \rho(\{v_1,v_2, v_4,v_5\}).
\end{array}
\right.
$$
Hence, $\PMSGd$ is not compressed, a contradiction.
Thus, $G'$ is a complete bipartite graph $K_{n,n}$.
\end{proof}

\begin{lemma}
\label{twotriangles}
Let $G$ be a connected graph.
If $\PMSG$ is compressed, then any two triangles of $G$
have a common edge.
\end{lemma}

\begin{proof}
Suppose that $\PMSG$ is compressed and two triangles $C$ and $C'$ of $G$
have no common edges.

\bigskip

\noindent
{\bf Case 1.} ($C$ and $C'$  have exactly one common vertex)

Let $G'=C \cup C'$, where $C=(v_1,v_2,v_3)$ and $C'=(v_3,v_4,v_5)$. 
Then the vertex set and the edge set of $G'$ are 
$$V'=\{v_1,v_2,v_3,v_4,v_5\},\ 
E'=\{\{v_1,v_2\},\{v_2,v_3\},\{v_3,v_1\},\{v_3,v_4\},\{v_4,v_5\},\{v_5,v_3\}\}.$$
Since $G$ has no odd cycle of length $\ge 5$ as a subgraph, $G'$ is an induced subgraph of $G$, and hence
$\PMSGd$ is compressed. 

We now consider the facets of $\PMSGd$.
Let $S=\{v_3\}$.
Since $|S|=1$, the set $S$ satisfies conditions (i) and (iii) in Lemma \ref{nonbipartitefacet}.
In addition, since $G' \setminus (S\cup\Gamma(S))$ is empty, $S$ satisfies condition (ii) in Lemma \ref{nonbipartitefacet}.
Thus, $S$ induces a facet of $\PMSGd$.
However, we have
$$
x(S)-x(\Gamma(S))
=
\left\{
\begin{array}{cl}
0     & \mbox{for } x =\rho(\emptyset) = {\bf 0}, \\
-2     & \mbox{for }  x = \rho(\{v_1, v_2\}),\\
-4     & \mbox{for } x = \rho(\{v_1,v_2, v_4,v_5\}).
\end{array}
\right.
$$
Hence, $\PMSGd$ is not compressed, a contradiction.
\bigskip

\noindent
{\bf Case 2.} ($C$ and $C'$ have no common vertices)

Since $G$ is connected, there exists a path $P=(v_1, p_1, \ldots, p_s = v_1')$ connecting two triangles $C=(v_1,v_2,v_3)$ and  $C'=(v'_1,v'_2,v'_3)$, where $p_1,\dots,p_{s-1} \notin V(C) \cup V(C')$.
We may assume that $s$ $(\ge 1)$ is minimal among pairs of triangles without common edges.
Let $G''$ be an induced subgraph on the vertex set $V(C)\cup V(C') \cup V(P)$.
Let $S=\{v_2\}$.  
Since $|S|=1$, the set $S$ satisfies conditions (i) and (iii) in Lemma \ref{nonbipartitefacet}.

\bigskip

\noindent
{\bf Case 2.1. } ($\Gamma(S) = \{v_1, v_3\}$)

For this case, $G'' \setminus(S\cup\Gamma(S))$ is nonbipartite. 
However, we have
$$
x(S)-x(\Gamma(S))
=
\left\{
\begin{array}{cl}
0     & \mbox{for } x =\rho(\emptyset) = {\bf 0}, \\
-1     & \mbox{for }  x = \rho(\{{v_1,  p_1}\}),\\
-2     & \mbox{for } x = \rho(\{v_1,v_3\}).
\end{array}
\right.
$$
Hence, $\Pc_{G''}$ is not compressed, a contradiction.

\bigskip

\noindent
{\bf Case 2.2. } ($\Gamma(S) \neq \{v_1, v_3\}$)

There exists an edge $\{v_2, v\}$, where $v$ $(\ne v_1)$ belongs to either 
the path $P$ or $C'$.
If $v\ne p_1$,
then an odd cycle of length $\ge 5$ is a subgraph of $G''$. 
Hence, we have $v=p_1$. 
Then $G''$ has triangles $(v_1,v_2,v)$ and $(v'_1,v'_2,v'_3)$ connected by a path $(v=p_1, \ldots, p_s = v_1')$ ($s \ge 2$).
This contradicts the hypothesis that $s$ is minimal.
\end{proof}

We are now in a position to prove a main theorem.

\begin{proof}[Proof of Theorem \ref{compressed}]
(``Only if")
Suppose that $\PMSG$ is compressed.
From Proposition \ref{lineperfect} and Lemma \ref{nooddcycle},
each block of $G$ is either a bipartite graph, $K_4$, or $K_{1,1,n}$.
By Lemma \ref{twotriangles}, 
at most one block is either $K_4$ or $K_{1,1,n}$.
It is enough to show that each bipartite block is a complete bipartite graph.
Let $B$ be a bipartite block of $G$ on the vertex set $B_1 \cup B_2$.
Suppose that $\{i,j\}$ is not an edge of $G$ for
vertices $i \in B_1$ and $j \in B_2$.
Since $B$ is 2-connected, there exist two disjoint paths
$P_1$ and $P_2$
from $i$ to $j$ in $B$.
Note that the length of each $P_i$ is at least 3.
Hence $P_1 \cup P_2$ is an even cycle of length $\ge 6$.
This contradicts to Lemma \ref{completebipartite}.
Thus, $B$ is a complete bipartite graph.

(``if")
Suppose that all blocks of $G$ are complete bipartite graphs except for at most one block, which is either $K_4$ or $K_{1,1,n}$
and $\PMSG$ is not compressed.

\bigskip

\noindent
{\bf Case 1.} ($G$ is bipartite)
There exists $\emptyset \ne S\subsetneq  V_1$ such that 
$$x(S)-x(\Gamma(S)) \le 0$$
is facet-inducing, and 
$$x(S)-x(\Gamma(S)) \le -2$$
for some $x$.
It then follows that there exist four distinct vertices $i,i'\in \Gamma (S)$ and $j, j' \in V_1-S$
such that $\{i,j\},\{i',j'\}\in E $.
By Proposition \ref{bipartitefacet},  
$G[S \cup \Gamma(S)]$ and $G[(V_1-S) \cup (V_2 -\Gamma(S))]$ are connected.
Then there exists an even cycle
$$C=(i,j, k_1,\ldots,k_{2p-1}, j',i',\ell_1,\dots,\ell_{2q-1}),$$
where 
$$k_1,k_3,\ldots,k_{2p-1}\in V_2 -\Gamma(S),\ \ \ 
k_2,k_4,\ldots,k_{2p-2}\in V_1-S,$$
$$\ell_1,\ell_3, \ldots,\ell_{2q-1}\in S, \ \ \ 
\ell_2,\ell_4, \ldots,\ell_{2q-2}\in \Gamma(S).$$
Note that the length of $C$ is at least $6$.
However, since $G[V(C)]$ is complete bipartite, $G$ has an edge $\{k_1,\ell_1\}$.
This contradicts $k_1\in V_2 -\Gamma(S)$ and
$\ell_1 \in S$.

\bigskip

\noindent
{\bf Case 2.} ($G$ is not bipartite)

There exists a subset $S\subset V$ such that
\begin{equation}
    x(S)-x(\Gamma(S)) \le |S| - \theta(S) \label{facet case2.1}
\end{equation}
is facet-inducing, and 
$$x(S)-x(\Gamma(S)) \le |S| - \theta(S) -2$$
for some $x$.
Note that $G$ has no odd cycle of length $\ge 5$ as a subgraph.

\bigskip

\noindent
{\bf Case 2.1.} ($S$ is not stable)
Since every component of $G[S]$ is critical, $G[S]$ has a triangle $T_1$.
By Lemma \ref{twotriangles}, $G\setminus(S\cup\Gamma(S))$ has no triangle as a subgraph.
Thus, $G\setminus(S\cup\Gamma(S))$ is bipartite if it is not empty.
Since (\ref{facet case2.1}) is facet-inducing,
we have $G=G[S \cup \Gamma(S)]$.
Furthermore, since $G$ has a matching satisfying $x(S)-x(\Gamma(S)) \le |S| - \theta(S)-2$,  
it follows that $G[\Gamma(S)]$ has an edge $\{i,j\}$.
Since the graph obtained from $G$ by deleting all edges with both ends in $\Gamma(S)$ is connected,
there exists a path $P$ from $i$ to $j$ 
which does not contain $\{i,j\}$.
Then $T_2 = P \cup \{i,j\}$ is an odd cycle of $G$.
Since the length of $T_2$ is 3, and $T_1$ and $T_2$ have no common edge, this is a contradiction.

\bigskip

\noindent
{\bf Case 2.2.} ($S$ is stable and $G \neq G[S \cup \Gamma(S)]$)
By Proposition \ref{nonbipartitefacet}, 
every component of $G\setminus(S\cup\Gamma(S))$ is nonbipartite, and hence has a triangle.
Since any two triangles of $G$ have a common edge,
it follows that $G\setminus(S\cup\Gamma(S))$ is connected.
If $G[\Gamma(S)]$ has an edge $\{i,j\}$, $G[S \cup \Gamma(S)]$ has a triangle as in Case 2.1.
This is a contradiction.
Hence, $\Gamma(S)$ is a stable set. 
Since $x(S)-x(\Gamma(S)) \le |S| - \theta(S)-2 = -2$
for some $x$, it follows that there exist four distinct vertices $i,i'\in \Gamma (S)$ and $j, j' \in V_1-S$
such that $\{i,j\},\{i',j'\}\in E $.
Since $G[S \cup \Gamma(S)]$ and $G[(V_1-S) \cup (V_2 -\Gamma(S))]$ are connected, the same argument as in Case 1 yields a contradiction.

\bigskip

\noindent
{\bf Case 2.3.} ($S$ is stable and $G = G[S \cup \Gamma(S)]$)
From $|S| = \theta(S)$, 
$$x(S)-x(\Gamma(S))$$
must be even number.
Since $\PMSG$ is not compressed, there exists $x$ such that 
$$x(S)-x(\Gamma(S)) \le -4.$$
Hence, $G[\Gamma(S)]$ has two edges without a common vertex and hence $G$ has two triangles without a common edge.
This is a contradiction.
\end{proof}

\begin{example}
The perfectly matchable subgraph polytope $\PMSG$ of the graph $G$ in Figure~\ref{compressednorei} is compressed.
\end{example}

\begin{figure}

\begin{tikzpicture}

\coordinate  (v1) at (1,0);
\coordinate (v2) at (1,-2);
\coordinate (v3) at (-1,-2);
\coordinate (v4) at (-1,0);
\coordinate (v5) at (3,0);
\coordinate (v6) at (5,0);
\coordinate (v7) at (2,2);
\coordinate (v8) at (4,2);
\coordinate (v9) at (-3,0);
\coordinate (v10) at (-5,0);
\coordinate (v11) at (-1,2);
\coordinate (v12) at (-3,2);
\coordinate (v13) at (-5,2);

\draw
(v1)--(v2)
(v1)--(v3)
(v1)--(v4)
(v2)--(v3)
(v2)--(v4)
(v3)--(v4)

(v1)--(v7)
(v1)--(v8)
(v5)--(v7)
(v5)--(v8)
(v6)--(v7)
(v6)--(v8)

(v4)--(v11)
(v4)--(v12)
(v4)--(v13)
(v9)--(v11)
(v9)--(v12)
(v9)--(v13)
(v10)--(v11)
(v10)--(v12)
(v10)--(v13)
;

\fill
(v1) circle (4pt)
(v2) circle (4pt)
(v3) circle (4pt)
(v4) circle (4pt)
(v5) circle (4pt)
(v6) circle (4pt)
(v7) circle (4pt)
(v8) circle (4pt)
(v9) circle (4pt)
(v10) circle (4pt)
(v11) circle (4pt)
(v12) circle (4pt)
(v13) circle (4pt)
;

\end{tikzpicture}
\caption{Graph which has $K_{3,3}$, $K_4$, and $K_{2,3}$ as blocks}
\label{compressednorei}
\end{figure}
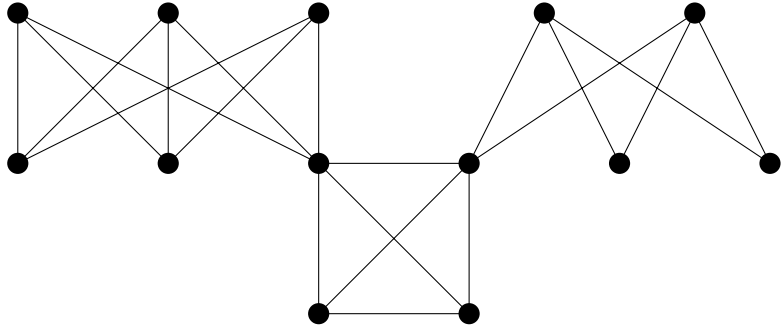

\section{Gorenstein perfectly matchable subgraph polytopes}

In this section, 
for several classes of graphs,
we give a characterization
of a graph $G$ such that $K[\PMSG]$ is Gorenstein.
If $G$ is either $K_1$ or $K_2$, then
$K[\PMSG]$ is isomorphic to a polynomial ring 
and hence $K[\PMSG]$ is Gorenstein.
Throughout this section, we may assume that $G$ has at least 
two edges.


\subsection{2-connected bipartite graphs}

Suppose that $G$ is a bipartite graph on the vertex set $V=V_1\cup V_2=\{1,\dots,n\}$, where $n\in V_2$. 
Then $\PMSG$ lies on the hyperplane $\mathcal{H}$ defined by the equation $x(V_1)=x(V_2)$. 
Let $\psi:\RR^{n-1}\rightarrow\mathcal{H}$ denote the affine map defined by setting 
$$\psi(y)=(y_1,\dots,y_{n-1},y(V_1)-y(V_2 \setminus \{n\})),$$
for each $y=(y_1,\dots, y_{n-1}) \in \RR^{n-1}$. 
Then $\psi$ is an affine isomorphism such that $\psi(\ZZ^{n-1})=\mathscr{H}\cap \ZZ^{n}$. 
Hence, $\psi^{-1}(\PMSG)\subset \RR^{n-1}$ is a lattice polytope of dimension $n-1$
which is unimodulary equivalent to $\PMSG$.

If $G$ is bipartite, we have the following criterion for $G$ whose $\PMSG$ is Gorenstein.
Note that any vertex of degree one is not a cut vertex.

\begin{prop}
\label{bipGoren}
Let $G$ be a connected bipartite graph on the vertex set $V=\{1,2,\dots,n\} = V_1 \cup V_2$.
Then $\PMSG$ is Gorenstein of index $\delta$ if and only if 
$\delta \ge 2$ and 
there exists $\alpha \in \ZZ^{n}$
such that the following hold:
\begin{itemize}
    \item[(i)]
    $\alpha(V_1) = \alpha(V_2)${\rm ;}
    \item[(ii)]
    If $v$ is not a cut vertex, then $\alpha(v) = 1${\rm ;}
    \item[(iii)]
    If $\deg(v) \ge 2$, then $\alpha(v) = \delta -1${\rm ;}
    \item[(iv)]
    If $G[S\cup\Gamma(S)]$ and $G[(V_1\setminus S)\cup(V_2\setminus \Gamma(S))]$ are connected for a subset $\emptyset \ne S\subsetneq V_1$,
    then $\alpha(S) - \alpha(\Gamma(S))=-1$.
\end{itemize}
\end{prop}

\begin{proof}
Let $\Pc = \psi^{-1}(\PMSG)$, where $\psi$ is the map defined as above.
Then $\Pc$ is Gorenstein of index $\delta$
if and only if there exists 
a lattice point $\beta \in\delta (\Pc\setminus\partial\Pc)\cap\ZZ^{n-1}$ such that $\delta \Pc - \beta$ is a reflexive polytope, where $\delta \Pc =\{\delta a:a\in\Pc\}$.

By Proposition \ref{bipartitefacetfor}, substituting $x(n) = x(V_1) - x(V_2\setminus \{n\}) $, it follows that 
$\delta \Pc$ is a set of vectors $x \in \RR^{V \setminus \{n\}}$ such that
\begin{eqnarray}
0\leq x(v) \leq \delta & \mbox{ for all }  v \in V\setminus \{n\}, \label{deltazeroichi}\\ 
0 \le x(V_1) - x(V_2\setminus \{n\}) \leq \delta,& \\ 
x(S)-x(\Gamma(S))\leq0 &\mbox{ for all } \emptyset \ne S\subsetneq
V_1 \mbox{ such that } n \notin \Gamma(S),
\label{deltafacet2}\\
-x(V_1 \setminus S) + x(V_2 \setminus \Gamma(S) )\leq0 &\mbox{ for all } \emptyset \ne S\subsetneq V_1
\mbox{ such that } n \in \Gamma(S).\label{newfact}
\end{eqnarray}
Furthermore, 
$\delta \Pc -\beta$, where $\beta\in\delta (\Pc\setminus\partial\Pc)\cap\ZZ^{n-1}$, is a set of vectors $x \in \RR^{V \setminus \{n\}}$ such that
$$
-\beta(v) \leq x(v) \leq \delta -\beta (v)  \mbox{ for all }  v \in V\setminus \{n\}, 
$$
$$
-\beta(V_1) + \beta(V_2 \setminus \{n\}) \le x(V_1) - x(V_2\setminus \{n\}) \leq \delta - \beta(V_1) + \beta(V_2 \setminus \{n\}), 
$$
$$
x(S)-x(\Gamma(S))\leq -\beta(S)+ \beta(\Gamma(S))\mbox{ for all } \emptyset \ne S\subsetneq V_1 \mbox{ such that } n \notin \Gamma(S),
$$
$$-x(V_1 \setminus S) + x(V_2 \setminus \Gamma(S) )\leq
\beta(V_1 \setminus S) - \beta(V_2 \setminus \Gamma(S) )\mbox{ for all } \emptyset \ne S\subsetneq V_1
\mbox{ such that } n \in \Gamma(S).
$$
%
%
%
%
%
%
Note that $0 < \beta(v) < \delta$ and $-\beta(S)+\beta(\Gamma(S)\setminus \{n\})> 0$
for all $\emptyset \ne S\subsetneq V_1$.
Thus, 
$\Pc$ is Gorenstein if and only if $\delta \ge 2$ and  there exists $\beta \in\delta (\Pc\setminus\partial\Pc)\cap\ZZ^{n-1}$
such that
$$
\begin{array}{rll}
-\beta(v) &=-1 & \mbox{ if } v \ (\ne n) \mbox{ is not a cut vertex}\\
\delta-\beta(v)&=1 & \mbox{ if } \deg(v)\geq2, v \ne n\\
-\beta(V_1) + \beta(V_2 \setminus \{n\})&=-1 & \mbox{ if } n \mbox{ is not a cut vertex}\\
\delta - \beta(V_1) + \beta(V_2 \setminus \{n\})&=1 & \mbox{ if } \deg(n)\geq2\\
-\beta(S)+\beta(\Gamma(S)) &= 1 & \mbox{ if } n \notin \Gamma(S) \mbox{ and } G[S\cup\Gamma(S)] \mbox{ and } G[(V_1\setminus S)\cup(V_2\setminus \Gamma(S))]\\
 &  & \mbox{ are connected for a subset } \emptyset \ne S\subsetneq V_1,\\
 \beta(V_1 \setminus S) - \beta(V_2 \setminus \Gamma(S) ) &=1 & \mbox{ if } n \in \Gamma(S) \mbox{ and } G[S\cup\Gamma(S)] \mbox{ and } G[(V_1\setminus S)\cup(V_2\setminus \Gamma(S))]\\
 & &  \mbox{ are connected for a subset } \emptyset \ne S\subsetneq V_1.
\end{array}
$$
By taking $\alpha = \psi (\beta)$, this is equivalent to 
conditions (i)--(iv).
%
%
%
\end{proof}

The following characterization 
is known for graphs having a perfect matching.

\begin{prop}[Hall's marriage theorem]
Let $G$ be a bipartite graph on the vertex set $V_1 \cup V_2$.
Then $G$ has a perfect matching if and only if 
$|V_1| = |V_2|$ and
$|S| \le |\Gamma(S)|$ for any subset $\emptyset \ne S \subset V_1$.
\end{prop}

Recall that a
lattice polytope $\Pc \subset \RR^n$ is said to be 
{\it Gorenstein of index $2$} 
if there exists a lattice point $\alpha \in 2 (\Pc\setminus\partial\Pc)\cap\ZZ^n$ such that $2 \Pc - \alpha$ is a reflexive polytope.

\begin{prop}

\label{sagaichi}
Let $G$ be a connected bipartite graph.
Then $\PMSG$ is Gorenstein of index $2$
 if and only if $G$ has a perfect matching and, for any subset $\emptyset \ne S\subsetneq V_1$ such that
$G[S\cup\Gamma(S)]$ and $G[(V_1\setminus S)\cup(V_2\setminus \Gamma(S))]$ are connected,
    we have $|S| +1 = |\Gamma(S)|$.
    
\end{prop}

\begin{proof}
Let $\alpha$ be the vector in Proposition \ref{bipGoren}.
Since $0<\alpha(v)< \delta = 2$ for any $v \in V_1 \cup V_2$, 
we have $\alpha = (1,1,\dots,1)$. Hence,
$\alpha(S) - \alpha(\Gamma(S))=|S| - |\Gamma(S)|$.
From Hall's marriage theorem,
the vector $\alpha= (1,1,\dots,1)$ belongs to $2\PMSG$
if and only if $G$ has a perfect matching.
\end{proof}

\begin{thm}
\label{2Gor}
Suppose that a connected bipartite graph
$G$ has a vertex $v$ with $\deg(v) \ge 2$ such that $v$ is not a cut vertex.
Then the following conditions are equivalent{\rm :}
\begin{itemize}
    \item[{\rm (i)}] 
$\PMSG$ is Gorenstein{\rm ;}
    \item[{\rm (ii)}] 
    $\PMSG$ is Gorenstein of index $2${\rm ;}
    \item[{\rm (iii)}] 
$G$ has a perfect matching and, for any subset $\emptyset \ne S\subsetneq V_1$ such that
$G[S\cup\Gamma(S)]$ and $G[(V_1\setminus S)\cup(V_2\setminus \Gamma(S))]$ are connected,
    we have $|S| +1 = |\Gamma(S)|$.
    
\end{itemize}

\end{thm}

\begin{proof}
From Proposition~\ref{sagaichi}, (ii) $\Leftrightarrow$ (iii).
It is obvious that (ii) $\Rightarrow$ (i).
It remains prove that (i) $\Rightarrow$ (ii).
Since $\alpha$ in Proposition \ref{bipGoren}
satisfies $\alpha(v) = \delta -1  =1$, we obtain $\delta=2$. 
\end{proof}

On the other hand, the following is known.

\begin{prop}[{\cite[Theorem 2.1 (iii) ($\mbox{a}'$)]{OH}}]
\label{edgepolyGor}
 Let $G$ be a 2-connected bipartite graph.
Then the edge polytope ${\rm Ed}(G)$ is Gorenstein if and only if $G$ has a perfect matching and, for any subset $\emptyset \neq S
\subset V_1$ such that
$G[S\cup\Gamma(S)]$ is connected and 
that $G[(V_1\setminus S)\cup(V_2\setminus \Gamma(S))]$ is connected and has at least one edge,
 we have $|S| +1 = |\Gamma(S)|$.
\end{prop}

If $G$ is a 2-connected bipartite graph, 
the conditions in Proposition~\ref{sagaichi}
and Proposition~\ref{edgepolyGor} 
are equivalent.
Since the statement is slightly different,
we give a proof for the readers.

\begin{cor}
\label{pmsedgeGor}
Let $G$ be a 2-connected bipartite graph.
Then $\PMSG$ is Gorenstein if and only if the edge polytope ${\rm Ed}(G)$ of $G$ is Gorenstein.
\end{cor}

\begin{proof}
Suppose that $G$ is a 2-connected bipartite graph.
Then $\deg(v) \ge 2$ for any vertex $v$ of $G$.

Suppose that ${\rm Ed}(G)$ is Gorenstein.
Assume that, for $\emptyset \ne S\subsetneq V_1$,
$G[(V_1\setminus S)\cup(V_2\setminus \Gamma(S))]$
is a connected graph with no edges.
Then
$G[(V_1\setminus S)\cup(V_2\setminus \Gamma(S))]$
has exactly one vertex $v$.
It then follows that  $V_1  = S \cup \{v\}$
and $V_2 = \Gamma(S)$.
Since $G$ has a perfect matching,
we have $|V_1| = |V_2|$.
Thus, $|S| +1 = |V_1|= |V_2|=|\Gamma(S)|$
in this case.
Thus, by Proposition~\ref{2Gor}, $\PMSG$ is Gorenstein.

Suppose that $\PMSG$ is Gorenstein.
Assume that, for $\emptyset \ne S\subset V_1$,
$G[(V_1\setminus S)\cup(V_2\setminus \Gamma(S))]$
is a connected graph.
If $S \ne V_1$, then $|S| +1 =|\Gamma(S)|$ by Proposition~\ref{2Gor}.
If $S=V_1$, then $\Gamma(S) = V_2$
and hence $G[(V_1\setminus S)\cup(V_2\setminus \Gamma(S))]$ has no vertices.
Thus, by Proposition~\ref{edgepolyGor}, ${\rm Ed}(G)$ is Gorenstein.
\end{proof}

Theorem~\ref{secondmain} follows from
Theorem \ref{2Gor} and Corollary \ref{pmsedgeGor}.

\begin{remark}
The conclusion of Corollary \ref{pmsedgeGor} is not true if
$G$ is not a $2$-connected bipartite graph.
There are many bipartite graphs $G$ such that ${\rm Ed}(G)$ is Gorenstein and
$\PMSG$ is not Gorenstein.
\begin{itemize}
    \item[(a)] 
Let $G$ be a bipartite pseudotree.
Then the edge polytope ${\rm Ed}(G)$ of $G$
is Gorenstein since the toric ring of ${\rm Ed}(G)$ is either 
isomorphic to a polynomial ring or a hypersurface.
\item[(b)]
Let $G$ be a bipartite pseudotree.
Then $\PMSG$ is Gorenstein if and only if 
$G$ satisfies either (i) or (iv) in Theorem \ref{pseudotreeGorensteinODD}.
For example, 
the perfectly matchable subgraph polytope of the 
bipartite pseudotree with the edge set 
$$\{ \{1,2\}, \{2,3\}, \{3,4\}, \{1,4\}\} \cup \{\{4,5\},\{5,6\},\dots, \{n-1,n\} \}$$
is not Gorenstein.    
\end{itemize}
\end{remark}

\begin{prop}
Let $G$ be a 2-connected bipartite graph.
Then we have the following.
\begin{itemize}
    \item[(a)] If $G$ is outerplanar, then $\PMSG$ is Gorenstein.
    
    \item[(b)] If $G$ is 4-connected and planar, then $\PMSG$ is not Gorenstein.
\end{itemize}

\end{prop}

\begin{proof}
Let $G$ be a 2-connected bipartite graph with $n$ vertices.

(a)
   Suppose that $G$ is outerplanar. 
   Then $G$ has an even cycle $(i_1,i_2,\dots,i_n)$ of length $n$ which corresponds to the outer face of $G$.
 Suppose that, for a subset $\emptyset \neq S\subsetneq V_1$, both $G[S\cup\Gamma(S)]$ and
 $G[(V_1\setminus S)\cup(V_2\setminus \Gamma(S))]$ are connected.
We may assume that 
$$S \cup \Gamma(S) =
\{i_1, i_2, \ldots, i_{k_1'}\} \cup
\{i_{k_2}, i_{k_2+1}, \ldots, i_{k_2'}\} \cup \cdots \cup 
\{i_{k_p}, i_{k_p+1}, \ldots, i_{k_p'}\},
$$
where $k_j'+1 < k_{j+1} $ for each $j = 1,2, \ldots, p-1$. 
Since $S \neq V_1$, we may assume that $k_p' <n$.

\medskip

\noindent
{\bf Case 1.} ($p=1$.)
If $i_1$ (resp.~$i_{k_1'}$) belongs to $S$,
then $i_n$ (resp.~$i_{k_1'+1}$) belongs to $\Gamma(S)$.
This is a contradiction.
Hence $i_1$ and $i_{k_1'}$ belong to $\Gamma(S)$.
Then $S= \{i_2, i_4, \dots, i_{k_1'-1 }\}$ and 
$\Gamma(S)= \{i_1, i_3, \dots, i_{k_1'} \}$.
Hence 
$|S| + 1 =  |\Gamma(S)|$.

\medskip

\noindent 
{\bf Case 2.} ($p\ge 2$.) 
Since $G[S\cup\Gamma(S)]$ is connected, there exists an edge $e_1 = \{i_\alpha, i_\beta\}$ of
$G[S\cup\Gamma(S)]$, where $1 \le \alpha \le k_1'$ and $ k_q \le \beta \le k_q'$ for some $2 \le q \le p$.
On the other hand, since $G[(V_1\setminus S)\cup(V_2\setminus \Gamma(S))]$ is connected,
there exists an edge $e_2 = \{i_\gamma, i_\delta\}$ of
$G[(V_1\setminus S)\cup(V_2\setminus \Gamma(S))]$, where $k_1' < \gamma  < k_q$ and $k_q' < \delta  \le n$.
Then $e_1$ and $e_2$ intersect in the drawing.
This contradicts that $G$ is outerplanar.

\bigskip

Thus, $G$ satisfies the condition in
Proposition~\ref{sagaichi}, and hence
$\PMSG$ is Gorenstein.

(b)
Suppose that $G$ is 4-connected and planar.
Let $S=\{v\}$, where $v \in V_1$ is a vertex of $G$.
Since $G$ is 4-connected, the degree of each vertex of $G$ is greater than or equal to $4$.
Hence, we have $|\Gamma(S)| \ge 4 > |S| +1$.
In addition, $G[S \cup \Gamma(S)]$ is a star graph
and hence connected.
Since $G$ is 4-connected and planar,
it is known \cite[Lemma 1]{4-connected} that
$G[(V_1\setminus S)\cup(V_2\setminus \Gamma(S))]$ is connected.
Hence, $G$ does not satisfy the condition in
Proposition~\ref{sagaichi}.
Thus, $\PMSG$ is not Gorenstein.
\end{proof}

\subsection{Pseudotrees with an even cycle}

In this subsection, we give a proof of Theorem \ref{pseudotreeGorensteinODD} for a pseudotree which has an even cycle.

\begin{proof}[Proof of Theorem \ref{pseudotreeGorensteinODD} (when $G$ has an even cycle $C$)]
Let $V= V_c \sqcup V_t \sqcup V_p$ be a partition of the vertex set $V$ of $G$, where $V_c=\{v_1,\dots,v_{2n}\}$ is
the vertex set of the even cycle $C =(v_1,v_2,\dots,v_{2n})$,
$V_t = \{v \in V\setminus V_c : \deg(v) > 1  \}$,
and 
$V_p = \{v \in V\setminus V_c : \deg(v) = 1  \}$.
%
The graph $G'$ obtained from $G$ by deleting edges of $C$ has $2n$ connected components.
Each of these connected components of $G'$ is a tree.
We regard each tree as a rooted tree whose root is a vertex in $C$.
Let $T_v$ be the  rooted subtree of such a rooted tree in $G'$ whose root is $v$.
See Figure~\ref{subtreenorei}.

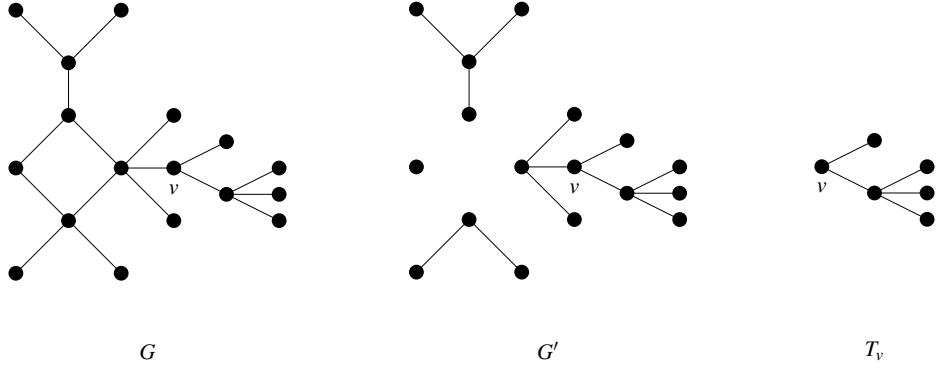
\begin{figure}
\scalebox{0.7}{
\begin{tabular}{cc}
\begin{minipage} [b] {0.5\hsize}
\begin{tikzpicture}
\coordinate (v1) at (0,1);
\coordinate (v2) at (1,0);
\coordinate (v3) at (0,-1);
\coordinate (v4) at (-1,0);
\coordinate (v5) at (2,1);
\coordinate (v6) at (2,0);
\coordinate (v7) at (2,-1);
\coordinate (v8) at (3,0.5);
\coordinate (v9) at (3,-0.5);
\coordinate (v10) at (4,0);
\coordinate (v11) at (4,-0.5);
\coordinate (v12) at (4,-1);

\coordinate (v13) at (2,-0.4) node at (v13) {$v$};

\coordinate (v14) at (1.5,-3.5) node at (v14) {$G$};

\coordinate (v15) at (0,2);
\coordinate (v16) at (1,3);
\coordinate (v17) at (-1,3);
\coordinate (v18) at (1,-2);
\coordinate (v19) at (-1,-2);

\draw
(v1)--(v2)
(v2)--(v3)
(v3)--(v4)
(v4)--(v1)
(v2)--(v5)
(v2)--(v6)
(v2)--(v7)
(v6)--(v8)
(v6)--(v9)
(v9)--(v10)
(v9)--(v11)
(v9)--(v12)
(v1)--(v15)
(v15)--(v16)
(v15)--(v17)
(v3)--(v18)
(v3)--(v19)
;

\fill 
(v1) circle (4pt)
(v2) circle (4pt)
(v3) circle (4pt)
(v4) circle (4pt)
(v5) circle (4pt)
(v7) circle (4pt)
(v8) circle (4pt)
(v9) circle (4pt)
(v10) circle (4pt)
(v11) circle (4pt)
(v12) circle (4pt)
(v6) circle (4pt)
(v15) circle (4pt)
(v16) circle (4pt)
(v17) circle (4pt)
(v18) circle (4pt)
(v19) circle (4pt)
;
\end{tikzpicture} 
\end{minipage}

\begin{minipage} [b] {0.5\hsize}
\begin{tikzpicture}
\coordinate (v1) at (0,1);
\coordinate (v2) at (1,0);
\coordinate (v3) at (0,-1);
\coordinate (v4) at (-1,0);
\coordinate (v5) at (2,1);
\coordinate (v6) at (2,0);
\coordinate (v7) at (2,-1);
\coordinate (v8) at (3,0.5);
\coordinate (v9) at (3,-0.5);
\coordinate (v10) at (4,0);
\coordinate (v11) at (4,-0.5);
\coordinate (v12) at (4,-1);

\coordinate (v13) at (2,-0.4) node at (v13) {$v$};

\coordinate (v14) at (1.5,-3.5) node at (v14) {$G'$};

\coordinate (v15) at (0,2);
\coordinate (v16) at (1,3);
\coordinate (v17) at (-1,3);
\coordinate (v18) at (1,-2);
\coordinate (v19) at (-1,-2);

\draw
(v2)--(v5)
(v2)--(v6)
(v2)--(v7)
(v6)--(v8)
(v6)--(v9)
(v9)--(v10)
(v9)--(v11)
(v9)--(v12)
(v1)--(v15)
(v15)--(v16)
(v15)--(v17)
(v3)--(v18)
(v3)--(v19)

;

\fill 
(v1) circle (4pt)
(v2) circle (4pt)
(v3) circle (4pt)
(v4) circle (4pt)
(v5) circle (4pt)
(v7) circle (4pt)
(v8) circle (4pt)
(v9) circle (4pt)
(v10) circle (4pt)
(v11) circle (4pt)
(v12) circle (4pt)
(v6) circle (4pt)
(v15) circle (4pt)
(v16) circle (4pt)
(v17) circle (4pt)
(v18) circle (4pt)
(v19) circle (4pt)
;
\end{tikzpicture} 
\end{minipage}

\begin{minipage} [b] {0.5\hsize}
\begin{tikzpicture}
\coordinate (v6) at (2,0);
\coordinate (v8) at (3,0.5);
\coordinate (v9) at (3,-0.5);
\coordinate (v10) at (4,0);
\coordinate (v11) at (4,-0.5);
\coordinate (v12) at (4,-1);

\coordinate (v13) at (2,-0.4) node at (v13) {$v$};

\coordinate (v14) at (3,-3.5) node at (v14) {$T_v$};

\draw
(v6)--(v8)
(v6)--(v9)
(v9)--(v10)
(v9)--(v11)
(v9)--(v12)
;

\fill 
(v8) circle (4pt)
(v9) circle (4pt)
(v10) circle (4pt)
(v11) circle (4pt)
(v12) circle (4pt)
(v6) circle (4pt)
;

\end{tikzpicture} 
\end{minipage}
\end{tabular}
}
\caption{Graph $G$ and the subtree which has $v$ as the root}
\label{subtreenorei}

\end{figure}

\bigskip

(``If")
Suppose that graph $G$ satisfies the condition that there exists an integer
$\delta \ge 2$ such that
$\deg (v) = \delta$ if $v \in V_c$ and $\deg (v) =\delta-1$ if $v \in V_t$.

By Proposition~\ref{bipGoren}, it is enough to
show that $\alpha\in\ZZ^n$, where 
$$
\alpha(v)=
\left\{
\begin{array}{cl}
\delta-1 & \mbox{if } v \in V_c \cup  V_t\\
1 & \mbox{if } v \in V_p
\end{array}
\right.
$$
satisfies conditions (i)--(iv) in Proposition~\ref{bipGoren}.

\bigskip

\noindent
(ii)
Note that $v$ is not a cut vertex
if and only if either 
(a) $\deg(v)=1$ or 
(b) $v \in V_c$ and $\deg(v)=2$ (then $G=C$).
In both cases, we have $\alpha(v)=1$.

\bigskip

\noindent
(iii)
If
$\deg(v)\ge 2$, then $v \in V_c \cup V_t$ and hence
we have $\alpha(v)=\delta -1$.

\bigskip

\noindent
(iv)
Suppose that 
$G[S\cup\Gamma(S)]$ and $G[(V_1\setminus S)\cup(V_2\setminus \Gamma(S))]$ are connected.
Then $S$ satisfies one of the following:

\bigskip

\noindent{\bf Case 1.}
($V_c \cap S = \emptyset$)
Since $G[S\cup\Gamma(S)]$ is connected, $G[S\cup\Gamma(S)]$ is $T_v$, where $v\in\Gamma(S)$.
Since $S$ and $\Gamma(S)$ give a partition of the vertex set of bipartite graph $G[S\cup\Gamma(S)]$,  the sum of the degree sequence of $S$ is equal to that of $\Gamma(S)$ 
in $G[S\cup\Gamma(S)]$.
However, $\deg_G (v) \ne \deg_{G[S\cup\Gamma(S)]} (v)$. 
If $v \in V_t$, $\deg_G (v) = \deg_{G[S\cup\Gamma(S)]} (v) +1$.
Since $(S \cup \Gamma(S)) \cap V_c = \emptyset$,
\begin{eqnarray*}
\alpha(S)-\alpha(\Gamma(S))&=&\sum_{v'\in S}\deg_{G}(v')-\sum_{v'\in \Gamma(S)}\deg_{G}(v')\\
&=&\sum_{v'\in S}\deg_{G[S\cup\Gamma(S)]}(v')-(1+\sum_{v'\in \Gamma(S)}\deg_{G[S\cup\Gamma(S)]}(v'))\\
&=&-1.
\end{eqnarray*}
If $v \in V_c$,
$\deg_G (v) = \deg_{G[S\cup\Gamma(S)]} (v) +2$.
Since $(S \cup \Gamma(S)) \cap V_c = \{v\}$,
\begin{eqnarray*}
\alpha(S)-\alpha(\Gamma(S))&=&\sum_{v'\in S}\deg_{G}(v')-(-1+\sum_{v'\in \Gamma(S)}\deg_{G}(v'))\\
&=&\sum_{v'\in S}\deg_{G[S\cup\Gamma(S)]}(v')-(-1+2+\sum_{v'\in \Gamma(S)}\deg_{G[S\cup\Gamma(S)]}(v'))\\
&=&-1.
\end{eqnarray*}

\bigskip

\noindent{\bf Case 2.}
($V_c \subset (S \cup \Gamma(S))$)
Similar to Case 1, $S$ and $\Gamma(S)$ are a partition of the vertex set of bipartite graph $G[S\cup\Gamma(S)]$, and the sum of the degree sequence of $S$ is equal to that of $\Gamma(S)$ in $G[S\cup\Gamma(S)]$.
Since $G[(V_1\setminus S)\cup(V_2\setminus \Gamma(S))]$ is connected, only an edge $\{i,j\}$ links $G[S\cup\Gamma(S)]$ and $G[(V_1\setminus S)\cup(V_2\setminus \Gamma(S))]$, where $i\in\Gamma(S),j\in V_1\setminus S$. This gives $\deg_G (i) = \deg_{G[S\cup\Gamma(S)]} (i) +1$. Hence, $\alpha(S)-\alpha(\Gamma(S))=-1$.

\bigskip

\noindent{\bf Case 3.}
(otherwise)
For $m \in \NN$, suppose that we have
$V_c \cap S =\{v_2,v_4,\dots,v_{2 m}\}$ and
$V_c \cap \Gamma(S) = \{v_1,v_3,\dots,v_{2m+1}\}$
by rearranging indices if necessary.
Then $G[S\cup\Gamma(S)]$ is a graph whose edge set is $\{ \{v_1,v_2\},\dots, \{v_{2m},v_{2m+1}\} \} \cup T_{v_1} \cup T_{v_2} \cup \dots \cup T_{v_{2m+1}}$.
By Case 1, $\alpha(S \cap T_{v_i})-\alpha(\Gamma(S \cap T_{v_i}))= (-1)^i$. 
If $T_{v_i}$ is $K_1$, then $\deg(v_i)=2$ and $\alpha(v_i)=1$.
Even if $v_i\in G[S\cup\Gamma(S)]$, we can treat it as a vertex which is a root of a tree.
Since $v_1,v_{2m+1}\in\Gamma(S)$,
$\alpha(S)-\alpha(\Gamma(S))=m-(m+1)=-1$.

\bigskip

In any case, we have $\alpha(S)-\alpha(\Gamma(S))=-1$. 

\bigskip

\noindent
(i) 
Let $T_{v_1}\neq K_1$ and
let $S$ be a subset of $V_1$ (or $V_2$) such that $G[S\cup\Gamma(S)]=T_{v_1}$.
By (iv) Case 1, $\alpha(S)-\alpha(\Gamma(S))=-1$. 
By (iv) Case 3, $\alpha(V_2 \setminus \Gamma(S))-\alpha(V_1 \setminus S)=-1$.
Hence, $\alpha(V_1)-\alpha(V_2)=0$.
If $T_1$ is $K_1$, then $\alpha(v_1)=1$.
Since $\alpha(V_1 \setminus {v_1})-\alpha(V_2)=-1$ by (iv) Case 3, we obtain $\alpha(V_1)-\alpha(V_2)=0$.

\bigskip

Hence, by Proposition~\ref{bipGoren}, $\PMSG$ is Gorenstein.

\bigskip

(``Only if")
Suppose that $\PMSG$ is Gorenstein.
Since any even cycle satisfies condition (iv) in Theorem~\ref{pseudotreeGorensteinODD},
we may assume that $G$ is not an even cycle.
Then we have $V_p \neq \emptyset$.
By Proposition~\ref{bipGoren}, there exist $\delta$ and $\alpha$ satisfying conditions (i)--(iv).
It then follows that $\delta\PMSG$ has 
$\alpha \in \ZZ^n$, where 
$$
\alpha(v)=
\left\{
\begin{array}{cl}
\delta-1 & \mbox{if } v \in V_c \cup  V_t\\
1 & \mbox{if } v \in V_p
\end{array}
\right.
$$
as an interior lattice point.

\bigskip

\noindent
{\bf Case 1.} ($V_c$ has a vertex which is not a cut vertex)
Since every vertex $v$ in $V_c$ satisfies $\deg(v)\ge2$, we obtain $\delta-1=1$ and
hence $\delta=2$, $\alpha=(1,\dots,1)$.
Since $G$ is not an even cycle, there exist
$S \subset V_p$ and $v \in V_t \cup V_c$ such that $\Gamma(S) = \{v\}$.
Since $G[S\cup\Gamma(S)]$ is $K_2$ or star graph,  $\alpha(S)-\alpha(\Gamma(S))\ge0$.
Both $G[S\cup\Gamma(S)]$ and $G[(V_1\setminus S)\cup(V_2\setminus \Gamma(S))]$ are connected.
From Proposition~\ref{bipGoren}, $\PMSG$ is not Gorenstein.
This is a contradiction.

\bigskip

\noindent
{\bf Case 2.} (any vertex in $V_c$ is a cut vertex)
Let $v\in V_t \cup V_c$.
Then $v$ is a cut vertex.
Let $S$ be a subset of $V_1$ (or $V_2$) such that $G[S\cup\Gamma(S)]$ is a subtree $T_v$ which has $v\in\Gamma(S)$ as a root.
Let $v'\in S$ be a child of $v$. 
If $\deg_G(v')\ge 2$, then there exists a subset $S' \subset \Gamma(S)$ such that $G[S'\cup\Gamma(S')]$ is a rooted subtree of $G[S\cup\Gamma(S)]$ with root $v'\in \Gamma(S')$ .
By condition (iv), $ \alpha(\Gamma(S'))-\alpha (S')= 1$.
If $\deg(v')=1$, then $\alpha(v')=1$.
Since $\alpha(v)=\delta-1$, 
$$-1=\alpha(S)-\alpha(\Gamma(S))=\sum_{v' \mbox{\tiny{ is a child of }} v} 1 -(\delta -1)= \deg_{G[S\cup\Gamma(S)]} (v)-\delta+1.$$
If $v \in V_c$, then $$\deg_G (v)=\deg_{G[S\cup\Gamma(S)]} (v) +2=\delta.$$
If $v \in V_t$, then $$\deg_G (v) = \deg_{G[S\cup\Gamma(S)]} (v) +1=\delta-1.$$
\end{proof}

\begin{example}
The perfectly matchable subgraph polytope $\PMSG$ of the graph $G$ in Figure~\ref{gorennorei} is Gorenstein.

\begin{figure}
\scalebox{0.8}{
\begin{tikzpicture}

\coordinate  (v1) at (1,1);
\coordinate (v2) at (1,-1);
\coordinate (v3) at (-1,-1);
\coordinate (v4) at (-1,1);
\coordinate (v5) at (3,1);
\coordinate (v6) at (1,3);
\coordinate (v7) at (5,1);
\coordinate (v8) at (3,3);
\coordinate (v15) at (5,3);
\coordinate (v16) at (7,1);
\coordinate (v9) at (3,-1);
\coordinate (v10) at (1,-3);
\coordinate (v11) at (-1,-3);
\coordinate (v12) at (-3,-1);
\coordinate (v13) at (-5,-1);
\coordinate (v14) at (-3,-3);
\coordinate (v17) at (-7,-1);
\coordinate (v18) at (-5,-3);
\coordinate (v19) at (-1,3);
\coordinate (v20) at (-3,1);
\coordinate (v21) at (-5,1);
\coordinate (v22) at (-3,3);

\draw
(v1)--(v2)
(v2)--(v3)
(v3)--(v4)
(v4)--(v1)
(v1)--(v5)
(v1)--(v6)
(v5)--(v7)
(v7)--(v15)
(v7)--(v16)
(v5)--(v8)
(v2)--(v9)
(v2)--(v10)
(v3)--(v11)
(v3)--(v12)
(v12)--(v13)
(v12)--(v14)
(v13)--(v17)
(v13)--(v18)
(v4)--(v19)
(v4)--(v20)
(v20)--(v21)
(v20)--(v22)
;

\fill
(v1) circle (4pt)
(v2) circle (4pt)
(v3) circle (4pt)
(v4) circle (4pt)
(v5) circle (4pt)
(v6) circle (4pt)
(v7) circle (4pt)
(v8) circle (4pt)
(v9) circle (4pt)
(v10) circle (4pt)
(v11) circle (4pt)
(v12) circle (4pt)
(v13) circle (4pt)
(v14) circle (4pt)
(v15) circle (4pt)
(v16) circle (4pt)
(v17) circle (4pt)
(v18) circle (4pt)
(v19) circle (4pt)
(v20) circle (4pt)
(v21) circle (4pt)
(v22) circle (4pt)

;

\end{tikzpicture}

}
\caption{$\deg(v_c)=4$ and $\deg(v_t)=3$}
\label{gorennorei}
\end{figure}
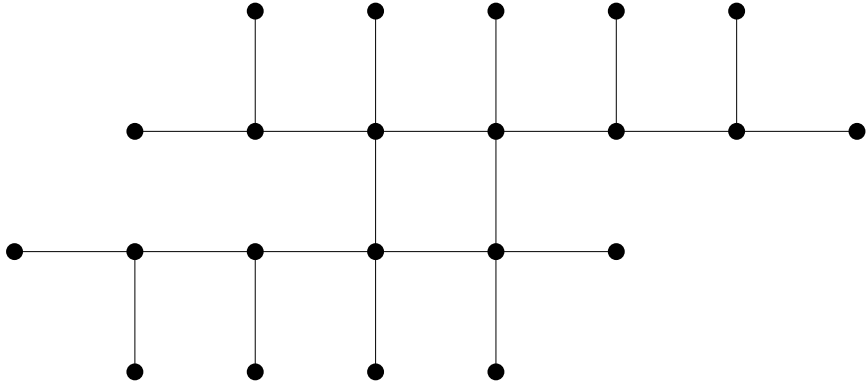

\end{example}




\subsection{Pseudotrees without even cycles}

In this subsection, we give a proof of Theorem \ref{pseudotreeGorensteinODD} for a pseudotree which has no even cycles.

A graph $G$ is said to be {\it $h$-perfect} if the stable set polytope $\stab (G)$ is defined by the constraints corresponding to cliques and odd holes, and the following nonnegativity constraints:
\begin{eqnarray}
x(v) \geq & 0 &  \mbox{ for any } v \in V(G)\\
\label{sspclique}
x(K) \leq &1 & \mbox{ for any maximal clique } K \mbox{ in } G\\
x(C) \leq&  n & \mbox{ for any odd cycle } C \mbox{ in } G \mbox{ of length } 2n+1
\end{eqnarray}
In particular, any perfect graph is $h$-perfect.
An {\it odd subdivision} of a graph $G$ is a graph obtained by replacing each edge of $G$ 
by a path of odd length. 
Let $C_5 + e$ be the graph obtained by adding a new edge to the cycle of length 5.
It is known \cite[Theorem 5]{Cao} that 
$L(G)$ is $h$-perfect if and only if $G$ has no odd subdivision of $C_5 + e$.
Since any pseudotree has at most one cycle, we have the following immediately.

\begin{lemma}
\label{peudotreeishperfect}
Let $G$ be a pseudotree.
Then $L(G)$ is $h$-perfect.
\end{lemma}




Let $K$ be a field. 
The \textit{Ehrhart ring} of a lattice polytope $P \subset \RR^n$ is 
$$K[{\xb}^{\alpha} s^{m}:\alpha\in mP\cap\ZZ^{n},m\in\ZZ_{\ge 0}]\subset K[x^{\pm}_1,\dots,x^{\pm}_n,s].$$
It is known that the Ehrhart ring of $P$
coincides with the toric ring of $P$
if and only if $P$ has IDP.
For the Ehrhart ring of $\stab(G)$, the following fact is known.

\begin{prop}[{\cite[Theorem 3.8]{Miy}}]
\label{hstabGor}
Let $G$ be an $h$-perfect graph.
Then the Ehrhart ring of $\stab (G)$ is Gorenstein if and only if 
all maximal cliques of $G$ have the same cardinality (say $\omega$),
and that $G$ satisfies one of the following conditions{\rm :}
\begin{itemize}
    \item[{\rm (i)}]
    $\omega=1${\rm ;}
    \item[{\rm (ii)}]
    $\omega=2$ and $G$ has no induced odd cycles of length $\ge 7${\rm ;}
    \item[{\rm (iii)}]
    $\omega \ge 3$ and $G$ has no induced odd cycles of length $\ge 5$.
\end{itemize}
\end{prop}

A {\it star} graph is a complete bipartite graph $K_{1,n}$.
Then $L(K_{1,n})$ is $K_n$. 
On the other hand, we have $L(C_3) = K_3$.
It is known that, in general, each clique in $L(G)$ corresponds to a star or to a triangle in $G$.
From this fact, we have the following.

\begin{proof}[Proof of Theorem \ref{pseudotreeGorensteinODD} (when $G$ has no even cycles)]
From Proposition~\ref{doukei}, 
$K[\PMSG]$ is Gorenstein if and only if $K[{\rm Stab}(L(G))]$ is Gorenstein.
In the proof of Proposition \ref{pseudotree normal},
we proved that ${\rm Stab}(L(G))$ is normal.
Since the lattice spanned by ${\rm Stab}(L(G)) \cap \ZZ^m $ is equal to $\ZZ^m$, 
${\rm Stab}(L(G))$ has IDP.
Hence, 
the toric ring of ${\rm Stab}(L(G))$ 
coincides with the Ehrhart ring of ${\rm Stab}(L(G))$.
Thus, $K[\PMSG]$  is Gorenstein if and only if 
the Ehrhart ring of ${\rm Stab}(L(G))$ is Gorenstein.

From Lemma \ref{peudotreeishperfect}, $L(G)$ is $h$-perfect.
Hence, by Proposition \ref{hstabGor}, 
the Ehrhart ring of ${\rm Stab}(L(G))$ is Gorenstein if and only if $L(G)$ satisfies one of conditions (i)--(iii) in Proposition~\ref{hstabGor}.
These conditions are
equivalent to the following, respectively:
\begin{itemize}
    \item[(i)]
$G$ is $K_1$ or $K_2$.
    \item[(ii)]
$G$ is either a path of length $\ge 2$ or $C_5$.
    \item[(iii)]
    $G$ is either a bidegreed tree which is not a path, or 
$G$ has a triangle $C$ where $\deg(v) \in \{2,3\}$ if $v \in V(C)$,
    and $\deg(v) \in \{1,3\}$ if $v \in V \setminus V(C)$.
\end{itemize}
Thus, (i)--(iii) above hold if and only if 
$K[\PMSG]$ is Gorenstein.
\end{proof}

\subsection{Complete multipartite graphs}

\begin{prop}
\label{completebipariteGor}
Let $G$ be a complete bipartite graph $K_{p,q}$ ($p\le q$). Then $K[\PMSG]$ is Gorenstein 
(equivalently, $\PMSG$ is Gorenstein) if and only if either $p=1$ or $p=q$.
\end{prop}

\begin{proof}
If $p=1$, then $G$ is a star graph. 
By Theorem \ref{pseudotreeGorensteinODD}, $\PMSG$ is Gorenstein.
Let $1<p\le q$. 
Then every vertex $v$ of $G$ is not a cut vertex and satisfies $\deg(v) \ge 2$. 
By Theorem~\ref{2Gor}, $\PMSG$ is not Gorenstein if $p \neq q$ since
$G$ has no perfect matchings.
If $p=q$, then $G$ has a perfect matching.
Since $G$ is complete bipartite, $\Gamma(S) = V_2$
for $\emptyset \ne S \subsetneq V_1$.
Hence $G[S\cup\Gamma(S)]$ and $G[(V_1\setminus S)\cup(V_2\setminus \Gamma(S))]$ are connected if and only if
$S=V_1\setminus\{v\}$ for some $v\in V_1$.
If $S=V_1\setminus\{v\}$, then $p=|S| +1 = |\Gamma(S)|$.
By Proposition~\ref{sagaichi}, $\PMSG$ is Gorenstein.
\end{proof}

Let $\Pc \subset \RR^n$ be a lattice polytope and let $\alpha \in \RR^n \setminus {\rm aff}(\Pc)$.
Then the convex hull of $\Pc \cup \{\alpha\}$ is called the \textit{pyramid} over $\Pc$ with apex $\alpha$.
In general, if $\Pc'$ is a pyramid over $\Pc$,
then $K[\Pc]$ is isomorphic to a polynomial ring in one variable over $K[\Pc']$.

\begin{prop}
\label{11n}
Let $G$ be a complete multipartite graph $K_{1,1,q}$ ($q \ge 1$).
Then $K[\PMSG]$ is Gorenstein if and only if $q \le 2$.
\end{prop}

\begin{proof}
Let $G=K_{1,1,q}$ on the vertex set
$\{1\} \cup \{2\} \cup \{3, 4, \dots, q+2\}$ and $G'=K_{2,q}$
on the vertex set
$\{1,2\} \cup \{3, 4, \dots, q+2\}$.
For $S=\{3,4,\ldots, q+2\}$, since every component of $G[S]$ is a single vertex,
and since the graph obtained from $G[S\cup\Gamma(S)]$ by deleting all edges with both ends in $\Gamma(S)$ is
the connected graph $G'$,
$$x(S)-x(\Gamma(S)) =x_3+x_4+\dots+x_{q+2}-(x_1+x_2)\le |S| - \theta(S)=q-q=0$$
is facet-inducing for $\PMSG$.
In addition, the facet is $\PMSGd$.
It then follows that $\PMSG$ is a pyramid over $\PMSGd$ with apex $\eb_1 + \eb_2$.
Hence the toric ring $K[\PMSG]$ is isomorphic to a polynomial ring in one variable over $K[\PMSGd]$.
Thus $K[\PMSG]$ is Gorenstein if and only if $K[\PMSGd]$ is Gorenstein.
From Proposition~\ref{completebipariteGor}, $K[\PMSGd]$ is Gorenstein if and only if $q \le 2$.
\end{proof}

\begin{prop}
\label{completeGor}
Let $G$ be a complete graph $K_n$.
Then $K[\PMSG]$ is Gorenstein if and only if $n \le 4$.
\end{prop}

\begin{proof}
By Theorem \ref{pseudotreeGorensteinODD},
$K[\PMSG]$ is Gorenstein if $n \le 3$.

Let $n \ge 4$.
By the definition of perfectly matchable subgraphs, 
we have 
$$\mathscr{W}(G) = \{ S \subset V : |S|  \equiv 0 \mbox{ (mod 2)}\},$$
since $G$ is complete.
Then $\PMSG$ is the {\it cut polytope} of a cycle $C_n$ of length $n$
(see \cite{Lau}) which is normal.
It is known \cite[Theorem 3.4]{OCut} that the toric ring of
the cut polytope of a graph $H$ is Gorenstein if and only if $H$ has no $K_5$-minor and satisfies one of the following:
\begin{itemize}
    \item[(i)] $H$ is a bipartite graph without induced cycle of length $\geq6$;
    \item [(ii)] $H$ is a bridgeless chordal graph.
\end{itemize}
Hence, the toric ring of the cut polytope of $C_n$ is Gorenstein if and only if $n \le 4$.
\end{proof}



It is an interesting problem to characterize complete multipartite graphs $G$ such that $K[\PMSG]$ is Gorenstein.
However, $\PMSG$ does not have IDP if $G$ is not bipartite. 
In addition, the normality of $\PMSG$ is unknown 
except for $K_{p,q}$, $K_{1,1,q}$, and $K_n$.

\vspace{0.5cm}
\leftline{\textbf{Availability of Data, Material and Code} Not applicable.}
\vspace{0.5cm}
\leftline{\textbf{\large Declaration}}
\vspace{0.5cm}
\leftline{\textbf{Conflict of Interest} The author declares that he has no conflict of interest.}

\section*{Acknowledgment}
The author is grateful to an anonymous referee for his/her careful reading and helpful comments.

\end{document}